\documentclass{article}

\usepackage[utf8]{inputenc}
\usepackage{tikz, pgf,amsthm,xr,mathrsfs}
\usepackage{amssymb,amsmath,caption,subcaption,graphicx,hyperref, xcolor}
\usepackage{amsthm}
\usepackage{authblk, blindtext}
\usepackage{lineno}
\usetikzlibrary{decorations.pathreplacing}

\usetikzlibrary{arrows, calc, matrix, patterns}

\pgfdeclarepatternformonly{South west lines}{\pgfqpoint{-0pt}{-0pt}}{\pgfqpoint{3pt}{3pt}}{\pgfqpoint{3pt}{3pt}}{
        \pgfsetlinewidth{0.4pt}
        \pgfpathmoveto{\pgfqpoint{0pt}{0pt}}
        \pgfpathlineto{\pgfqpoint{3pt}{3pt}}
        \pgfpathmoveto{\pgfqpoint{2.8pt}{-.2pt}}
        \pgfpathlineto{\pgfqpoint{3.2pt}{.2pt}}
        \pgfpathmoveto{\pgfqpoint{-.2pt}{2.8pt}}
        \pgfpathlineto{\pgfqpoint{.2pt}{3.2pt}}
        \pgfusepath{stroke}
}
\pgfdeclarepatternformonly{South east lines}{\pgfqpoint{-0pt}{-0pt}}{\pgfqpoint{3pt}{3pt}}{\pgfqpoint{3pt}{3pt}}{
    \pgfsetlinewidth{0.4pt}
    \pgfpathmoveto{\pgfqpoint{0pt}{3pt}}
    \pgfpathlineto{\pgfqpoint{3pt}{0pt}}
    \pgfpathmoveto{\pgfqpoint{.2pt}{-.2pt}}
    \pgfpathlineto{\pgfqpoint{-.2pt}{.2pt}}
    \pgfpathmoveto{\pgfqpoint{3.2pt}{2.8pt}}
    \pgfpathlineto{\pgfqpoint{2.8pt}{3.2pt}}
    \pgfusepath{stroke}}
    
\newcommand*{\GetListMember}[2]{%
   \edef\dotheloop{%
   \noexpand\foreach \noexpand\a [count=\noexpand\i] in {#1} {%
      \noexpand\IfEq{\noexpand\i}{#2}{\noexpand\a\noexpand\breakforeach}{}%
   }}%
   \dotheloop
   \par%
}%

\newcommand{\aw}{\operatorname{aw}}

\usepackage{enumitem}

\usepackage{tikzsymbols}
\usepackage{textcomp}
\usepackage{parskip}
\usepackage{float}
\usepackage[normalem]{ulem}

\newtheorem{thm}{Theorem}[section]
\newtheorem{defn}[thm]{Definition}
\newtheorem{prop}[thm]{Proposition}
\newtheorem{cor}[thm]{Corollary}
\newtheorem{lem}[thm]{Lemma}

\newtheorem{ex}[thm]{Example}

\newcommand{\bpf}{\begin{proof}}
\newcommand{\epf}{\end{proof}}

\newcommand{\rb}{\operatorname{rb}}

\newcommand{\mn}{[m]\times[n]}
\newcommand{\DD}{\mathbb{D}}

\definecolor{purple}{rgb}{0.64 ,0.17, 1.0}
\definecolor{MyGreen}{rgb}{0.15, 0.5, 0.06}
\definecolor{BurntOrange}{rgb}{0.7,0.35,0.0}
\definecolor{pink}{rgb}{1.0, 0.1, 0.5}
\definecolor{cbblue}{RGB}{68,119,170}
\definecolor{cbcyan}{RGB}{102,204,238}
\definecolor{cbgreen}{RGB}{34,136,51}
\definecolor{cbyellow}{RGB}{204,187,68}
\definecolor{cbred}{RGB}{238,102,119}
\definecolor{cbpurple}{RGB}{170,51,119}
\definecolor{cbgray}{RGB}{187,187,187}

\title{Rainbow numbers of $[m] \times [n]$ for $x_1+x_2 = x_3$}

\author[1]{Kean Fallon}
\author[2]{Ethan Manhart}
\author[1]{Joe Miller}
\author[3]{Hunter Rehm}
\author[2]{Nathan Warnberg}
\author[1]{Laura Zinnel}

\affil[1]{Department of Mathematics, Iowa State University, \{keanpf, jmiller0, lezinnel\}@iastate.edu}

\affil[2]{Department of Mathematics and Statistics, University of Wisconsin-La Crosse, ethan.manhart@gmail.com, nwarnberg@uwlax.edu}

\affil[3]{Department of Mathematics and Statistics, University of Vermont, hunter.rehm@uvm.edu}

\date{\today}

\begin{document}

\maketitle
\section*{Abstract}

Consider the set $[m]\times [n] = \{(i,j)\, : 1\le i \le m, 1\le j \le n\}$ and the equation $x_1+x_2 = x_3$, namely $eq$.  
The \emph{rainbow number of $[m] \times [n]$ for $eq$}, denoted $\rb([m]\times [n],eq)$, is the smallest number of colors such that for every surjective $\rb([m]\times[n], eq)$-coloring of $[m]\times [n]$ there must exist a solution to $eq$, with component-wise addition, where every element of the solution set is assigned a distinct color. 
This paper determines that $\rb(\mn, eq) = m+n+1$ for all values of $m$ and $n$ that a greater than or equal to $2$.

{\bf Keywords:} rainbow, anti-Ramsey, Schur

\section{Introduction}\label{sec:intro}

Given a set $S$, a coloring of $S$ assigns each element a color. Ramsey theory is the study of guaranteeing monochromatic structures in $S$.  
On the other hand, anti-Ramsey theory is the study of guaranteeing polychromatic (or rainbow) structures in $S$ and has gained the interest of many authors when $S$ is $[n] = \{1,2,\dots, n\}$ or $\mathbb{Z}_n$ \cite{ATHNWY, AF, BKKTTY, RFC, J, LM}. 
As an example, the anti-van der
Waerden number on $[n]$, denoted $\aw([n], k)$, is the smallest number
of colors such that every exact~(onto) $\aw([n], k)$-coloring of $[n]$ is guaranteed to have an
arithmetic progression of length $k$ where each element of the progression is colored
distinctly, \cite{BSY,DMS, U}.  The anti-van der Waerden number has also been studied in graphs, \cite{SWY, MW, RSW}, on finite abelian groups, \cite{finabgroup}, and were generalized further on $\mathbb{Z}_n$ for linear equations,~\cite{ATHNWY, BKKTTY, RFC, LM}. 
This inspired the authors in \cite{FGRWW} to look into rainbow numbers of $[n]$ for linear equations.

The rainbow number of $[n]$ for $eq$, denoted $\rb([n], eq)$, is the smallest number of colors such that for every exact
$\rb([n], eq)$-coloring of $[n]$, there exists a solution to $eq$ with every member of the
solution set assigned a distinct color. 
In this paper, the rainbow number of $[n]\times [m]$ for equation $x_1+x_2=x_3$ is solved completely.
Section~\ref{sec:genlem} establishes vocabulary, contains preliminary results and provides Examples and Figures that are referenced throughout the paper. 
Each section following contains a case analysis that ultimately leads to a complete solution. 
Section~\ref{sec:maindiagthree} analyzes $(m+n+1)$-colorings of $\mn$ when the main diagonal has three colors.
The analysis continues into Section \ref{sec:nojumps} when the main diagonal has four or more colors.  
Finally, Section \ref{sec:L} discusses some forbidden structures in rainbow-free $(m+n+1)$-colorings of $\mn$ and completes the analysis.
\section{Preliminaries}\label{sec:genlem}

 An \emph{$r$-coloring} of a set $S$ is a function $c:S \to [r]$, where $[r] = \{1,2,\dots,r\}$, and an $r$-coloring is \emph{exact} if it is surjective. 
 If $X\subseteq S$, then $c(X) = \{c(x)\,:\, x\in X\}$.  
 Given an equation $eq$, a solution in $S$ to $eq$ is a subset of elements of $S$ that satisfy the equation.  
 A solution $s$ is a \emph{rainbow solution to $eq$ with respect to coloring $c$} if the colors of the elements of $s$ are pairwise distinct.  If context is clear the reference to the coloring $c$ will be dropped and `$s$ is a rainbow solution' will often be used.  The \emph{rainbow number of $S$ for $eq$}, denoted $\rb(S,eq)$, is the smallest number such that every exact $\rb(S,eq)$-coloring of $S$ contains a rainbow solution to $eq$.  A coloring of $S$ that has no rainbow solutions to $eq$ is called \emph{rainbow-free}.  Although solutions to $eq$ are sets, they will often be thought of as lists.  That is, if it is claimed that $\{\alpha, \beta,\gamma\}$ is a solution then, in most cases, $\alpha + \beta = \gamma$.  Solutions with repeated elements cannot be rainbow and will be referred to as \emph{degenerate}.  If a set $S$ has no solutions to $eq$ then the convention is that $\rb(S,eq) = |S| + 1$.
 
Throughout the paper the set $S$ that will be discussed is $[m]\times [n] = \{(i,j) : i,j\in \mathbb{Z},  1\le i \le m \text{ and } 1\le j \le n\}$, with $m\le n$, and the equation, $eq$, is $x_1 + x_2 = x_3$.  Note that addition is component-wise and that if $m=1$ there are no solutions so, following convention, $\rb([1]\times [n]) = n+1$.

Define the \emph{$k^{th}$ diagonal} $D_k$ as the set 

\[D_k = \big\{(i,j)\in [m]\times[n]\,:\, m-k=i-j\big\}.\]

For reference, the authors view the array with $m$ rows and $n$ columns and the upper left element as $(1,1)$ thus, $D_1 = \{(m,1)\}$ and is in the bottom left corner.  The diagonal $D_m$ is called the \emph{main diagonal}.  A diagonal that is not the main diagonal will be referred to as an \emph{off-diagonal}.  Observe that $[m]\times[n]$ has $m + n - 1$ diagonals. 

 Definition \ref{defn:si} was established by the authors in \cite{FGRWW}.

 \begin{defn}\cite{FGRWW}\label{defn:si} Let $c$ be an exact $r$-coloring of $[n]$.  
Define $\mathcal{C}_i = \{a\in [n]\,:\,c(a) = i\}$ and $s_i\in \mathcal{C}_i$ such that $s_i$ is the smallest element of $\mathcal{C}_i$ for each color $i$.  Note that for any exact $r$-coloring $c$, it is always possible to have $s_i < s_j$ for $i < j$.
\end{defn}
 A modified definition is introduced for the purposes of this paper.  
 Let $c$ be a coloring of $[m]\times [n]$.  Without loss of generality, it can and will be assumed that if $|c(D_m)|=\ell$, then $c(D_m) = \{1,2,\dots,\ell\}$.  
 Now, define $s_1 = 1$ and for $k \geq 2$, define
\[\displaystyle
    s_k = \min_{1 \le x \le m}\big\{x: c((x,x)) \neq c((s_j,s_j)) \textrm{ for all $j < k$}\big\}.\]
 For any exact $r$-coloring $c$, it is always possible to have $s_i < s_j$ for $i < j$. If that is not the case, say $s_i > s_j$ and $i < j$, an isomorphic coloring can be created by swapping the color of any element with color $i$ to have color $j$ and vice versa.  Thus, it will also be assumed that $s_i < s_j$ when $i < j$.

In \cite{FGRWW}, rainbow numbers for the set $[m]$ for equation $x_1+x_2=x_3$ are investigated.  Note that that main diagonal of $\mn$ behaves similarly to $[m]$, thus results from \cite{FGRWW} can be applied to the main diagonal. Further, $s_i$ in Lemma \ref{FGRWWrainbowlem} is defined similarly to the $s_i$ established in this paper.

The following four results are from \cite{FGRWW} or extensions of results from \cite{FGRWW} to $[m]\times [n]$.  In particular, Lemma \ref{lem:s2} limits the number of colors that can appear in $D_m$.  The wording has been modified to match the language used in this paper.

\begin{lem}\label{FGRWWrainbowlem}\cite{FGRWW} Let $c$ be an exact, rainbow-free $r$-coloring of $[m]$ for $x_1+x_2 = x_3$, with color set $\{0,1,\dots, r-1\}$, then
    \begin{enumerate}
        \item if $s_i = \ell$, then $2\ell \le s_{i+1}$ for $0\le i \le r-2$,
        \item $2^i \le s_i$ for $0\le i \le r-1$.
        
    \end{enumerate}
\end{lem}

\begin{cor}\label{cor:2power}
    If $c$ is a rainbow-free coloring of $[m]$ and $\ell = |c([m])|$, then $2^{i-2}s_2 \le s_i$ for $2\le i \le \ell$. In particular, $2^{\ell-2}s_2 \le m$.
\end{cor}

\bpf The proof will proceed by induction on $i$.  The base case, when $i=2$, is clear.  As the inductive hypothesis, assume for some $2\le j \le \ell-1$ that $2^{j-2}s_2 \le s_j$ which gives $2^{j-1}s_2 \leq 2s_j$.  Further, by Lemma \ref{FGRWWrainbowlem} part 1, $2s_j \le s_{j+1}$. Therefore, $2^{j-1}s_2 \le s_{j+1}$ which completes the induction.\epf
    
\begin{thm}\label{FGRWWrainbowthm}\cite{FGRWW}
For $m\ge 3$, $\rb([m],x_1+x_2 = x_3) = \left\lfloor \log_2(m) + 2\right\rfloor$.
\end{thm}

\begin{lem}\label{lem:s2}
    If $c$ is a rainbow-free coloring of $\mn$ for $eq$ with $m\le n$, then $|c(D_m)| \le \left\lfloor \log_2(m) + 1\right\rfloor$ and $|c(D_m)| \le \log_2(m/s_2) + 2 $.
\end{lem}

\bpf 
Note that solutions to $x+y=z$ in $D_m$ correspond bijectively to solutions to $x+y=z$ in $[m]$, so Theorem \ref{FGRWWrainbowthm} can be applied to $D_m$.  This gives $|c(D_m)| \le \left\lfloor \log_2(m) + 1\right\rfloor$ immediately.  Further, define $\ell = |c(D_m)|$ and note that Corollary \ref{cor:2power} implies $2^{\ell -2}s_2 \le m$, thus $\ell \le \log_2(m/s_2) + 2$, as desired.\epf

Lemma \ref{mnlower} provides an exact, $(m+n)$-coloring that avoids rainbow solutions which establishes a lower bound on $\rb(\mn, eq).$  
The remaining results analyze the appearance of solutions to $eq$ in $\mn$ and the structure of a rainbow-free coloring on $\mn$ to help establish an upper bound on $\rb(\mn, eq)$.

\begin{lem}\label{mnlower}
For $2\le m\le n$, $m+n+1 \le \rb(\mn,eq)$.
\end{lem}

\begin{proof}Let $c:\mn \to [m+n]$ be defined by
\[c((i,j))=\left\{\begin{array}{ll}
                          1 & \text{if $ i<m$ and $j< n$,}\\
                          i+1 & \text{if $i<m\text{ and } j=n$,}\\
                          j+m & \text{if $i=m$.}
                             
                             \end{array}\right.\]
                             
Note that $(1,n)$ and $(m,1)$ are not in any solution to $eq$. If $\alpha,\beta,\gamma\in\mn$ such that $\alpha + \beta = \gamma$, then $c(\alpha)=c(\beta)=1$. Therefore, $c$ is rainbow-free for $eq$ and $m+n+1\leq \rb([m]\times[n],eq).$
\end{proof}

\begin{lem}\label{onedistinctcolorlem}  If $c$ is a rainbow-free coloring of $[m]\times [n]$ for $eq$ with $m\le n$, then, for all $D_x$ with $x\neq m$, $\big|c(D_x)\setminus c\big(D_m\big)\big|\leq 1$.
\end{lem}

\bpf If $m=1$ each diagonal has one element and the result follows.  Thus, assume $2 \le m$.  Since $|D_1| = |D_{m+n-1}| = 1$, it is certainly true that $|c(D_1) \setminus c(D_m)| \leq 1$ and $|c(D_{m+n-1}) \setminus c(D_m)|\leq 1$. 

For the purpose of contradiction, assume $1 < x < m+n-1$ and $|c(D_x)\setminus c(D_m)|\geq 2$. Then there exists $\beta$, $\gamma \in D_x$ such that $c(\beta),c(\gamma)\in c(D_x)\setminus c(D_m)$. However, there exists $\alpha \in D_m$ such that $\{\alpha ,\beta,\gamma\}$ is a rainbow solution, a contradiction. Thus, $\big|c(D_x)\setminus c\big(D_m\big)\big|\leq 1$.
\epf

Corollary \ref{cor:m=2} follows quickly from Lemmas \ref{mnlower} and \ref{onedistinctcolorlem}.  This allows most of the rest of the paper to focus on the situation when $3 \le m \le n$.

\begin{cor}\label{cor:m=2}
 If $m=2$ and  $m\le n$, then $\rb([m] \times [n]) = m+n+1$.
    
\end{cor}

\begin{cor}\label{mainlessthan3}
If $c$ is an exact, rainbow-free $(m+n+1)$-coloring of $[m]\times[n]$ for $eq$ with $3\le m \le n$, then $|c(D_m)| \geq 3$.
\end{cor}

\bpf
Lemma \ref{onedistinctcolorlem} implies $\big|c(D_x)\setminus c(D_m)\big|\leq 1$ for all $x \neq m$. So, $\big|c\big([m]\times[n]\big)\setminus c(D_m)\big|\leq m+n-2$. Thus, $|c(D_m)|\geq 3$. 
\epf

Let $c$ be a coloring of $\mn$.  Diagonal $D_j$ \emph{contributes color $x$} if $x\in c(D_j)\backslash c(D_m)$ and $x\notin c(D_i)$ for all $i < j$.  Otherwise, $D_j$ does not contribute color $x$.  In general, if $D_j$ contributes any color it is a \emph{contributing diagonal} and if it does not contribute any color it is a \emph{non-contributing diagonal}.

Lemma \ref{lem:ellminusthree} and the proceeding corollary establish results related to contributing off-diagonals. In particular, they provide an upper bound on the number of non-contributing off-diagonals and, subsequently, a lower bound on the number of contributing off-diagonals. This furthers the restrictions that applying an exact, rainbow-free $(m+n+1)$-coloring has on $[m]\times[n]$ when the main diagonal has exactly three colors, as in Lemma \ref{lem:eachdiagcontributes}, when the main diagonal has exactly four or more colors, as in Theorem \ref{thm:nojumpsever}, and more generally in Lemma \ref{lem:consecutivecontributing}.

\begin{lem}\label{lem:ellminusthree}
If $c$ is an exact, rainbow-free $(m+n+1)$-coloring of $\mn$ with $3\le m\le n$, then there are at most $|c(D_m)| - 3$ off-diagonals that do not contribute a color. 
\end{lem}

\bpf Define $\ell = |c(D_m)|$ and note that there are $m+n-2$ off-diagonals and $m+n+1 - \ell$ colors that appear in the off-diagonals that do not appear in the main diagonal.  Let $k$ be the number of off-diagonals that do not contribute a color.  For the sake of contradiction, assume $k \ge \ell - 2$.  This implies $m+n - 2-k \le m+n -\ell$ off-diagonals must contribute a color.  Thus, by the pigeon hole principle, some off-diagonal contains two colors that do not appear in the main diagonal. Since $c$ was assumed to be rainbow-free, Lemma \ref{onedistinctcolorlem} is contradicted.   Therefore, $k\le \ell -3$.\epf

\begin{cor}\label{cor:contributingoffdiag} If $c$ is an exact, rainbow-free $(m+n+1)$-coloring of $\mn$ with $3\le m\le n$, then there are at least $m+n-1-\log_2(m/s_2)$ contributing off-diagonals.

\end{cor}

\bpf Define $\ell=|c(D_m)|$ and observe that Lemma \ref{lem:ellminusthree} indicates that there are at most $\ell-3$ non-contributing off-diagonals.  Since there are $m+n-2$ off-diagonals in total, and off-diagonals are either contributing or non-contributing, there are at least $(m+n-2) -(\ell - 3) = m+n-\ell +1$ contributing off-diagonals.  Since Lemma \ref{lem:s2} implies $\ell \le \log_2(m/s_2)+2$, there are at least $m+n - \log_2(m/s_2) -1$ contributing off-diagonals.\epf


Let $\alpha = (a_1,a_2)\in D_a$ and $\beta =(b_1,b_2)\in D_b$ with $a\neq b$. If $a_1<b_1$ and $a_2<b_2$, there is a \emph{jump from $\alpha$ to $\beta$} and the \emph{jump distance} is defined to be $(b_1-a_1)+(b_2-a_2)$. 
Alternatively, there is a jump from $\alpha$ to $\beta$ when there exists some $\delta\in \mn$ such that $\alpha+ \delta = \beta$, where $\delta =(d_1,d_2)$ is called \emph{the jump} from $\alpha$ to $\beta$ and has jump distance $d_1 + d_2$.

\begin{ex}\label{ex1} In Figure \ref{fig:jump def}, $\alpha = (2,7)$, $\beta = (4,11)$, $\gamma = (5,2)$ and $\tau = (7,3)$.  The jump from $\alpha$ to $\beta$ is $\delta = (2,4)$ and has jump distance $6$.  The jump from $\gamma$ to $\tau$ is $(2,1)$ and has jump distance $3$.

\end{ex}

When there is a jump from $\alpha = (a_1,a_2)$ to $\beta = (b_1,b_2)$ with $\alpha \in D_a$ and $\beta \in D_b$, certain diagonals exhibit the special property that all elements within the diagonals make additional jumps with $\alpha$ or $\beta$. This set  of diagonals is 
\begin{equation}\label{eqS}S := \{D_x ~|~ m+a_2-b_1<x<m+b_2-a_1 \text{ and } x\notin\{a,b,m\}\}.\end{equation}

To visualize this, a rectangle is drawn using $\alpha$ and $\beta$ as corners. The diagonals with indices between $m+a_2-b_1$ and $m+b_2-a_1$ are those that intersect this rectangle not including the lower left and upper right corners. To finalize $S$, simply remove the diagonals containing $\alpha$ and $\beta$ and the main diagonal if applicable.

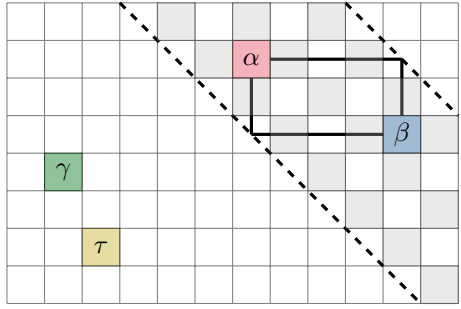
\begin{figure}[H]
    \centering
    
    \begin{tikzpicture} 
        \draw[step=.5cm,gray,thin] (0,0) grid (6,4);
        
        \filldraw[fill=cbred, draw=black, opacity = .5] (3,3) rectangle (3.5,3.5);
        \node at (3.25,3.25) {$\alpha$};
        \filldraw[fill=cbblue, draw=black, opacity = .5]  (5,2) rectangle (5.5,2.5);
        \node at (5.25,2.25) {$\beta$};
        
        \filldraw[fill=cbgreen, draw=black, opacity = .5] (.5,1.5) rectangle (1,2);
        \node at (.75,1.75) {$\gamma$};
        \filldraw[fill=cbyellow, draw=black, opacity = .5] (1,.5) rectangle (1.5,1);
        \node at (1.25,.75) {$\tau$};
        
        \draw[very thick] (3.25,3) -- (3.25,2.25);
        \draw[very thick] (3.25,2.25) -- (5,2.25);
        \draw[very thick] (3.5,3.25) -- (5.25,3.25);
        \draw[very thick] (5.25,3.25) -- (5.25,2.5);
        \draw[very thick, dashed] (1.5,4) -- (5.5,0);
        \draw[very thick, dashed] (4.5,4) -- (6,2.5);
        
        \foreach \n in {0,...,7}
            {
                \filldraw[fill=cbgray, draw = black, opacity = .3] (2+.5*\n,4-.5*\n) rectangle (2.5+.5*\n,3.5-.5*\n);
            }
        \foreach \n in {0,...,5}
            {
                \filldraw[fill=cbgray, draw = black, opacity = .3] (3+.5*\n,4-.5*\n) rectangle (3.5+.5*\n,3.5-.5*\n);
            }
        \foreach \n in {0,...,3}
            {
                \filldraw[fill=cbgray, draw = black, opacity = .3] (4+.5*\n,4-.5*\n) rectangle (4.5+.5*\n,3.5-.5*\n);
            }
    \end{tikzpicture}
    
    \caption{Jumps from $\alpha$ to $\beta$ and from $\gamma$ to $\tau$ from Example \ref{ex1}.  Also, a visualization of the set of diagonals $S$ (in gray) with respect to $\alpha$ and $\beta$.}
    \label{fig:jump def}
\end{figure}

The next two lemmas will show that if $\gamma$ is in a diagonal of $S$, then either $\alpha$ makes a jump to $\gamma$ or $\gamma$ makes a jump to $\beta$. Lemma \ref{lem:inside S} will show this for all diagonals of $S$ between $D_a$ and $D_b$, and Lemma \ref{lem:outside S} will consider the remainder of the diagonals in $S$.

\begin{lem}\label{lem:inside S}Suppose there is a jump from $\alpha=(a_1,a_2) \in D_a$ to $\beta=(b_1,b_2) \in D_b$.  If $\gamma = (g_1,g_2) \in D_g$ such that $a<g<b$ or $b<g<a$, then either $\alpha$ makes a jump to $\gamma$ or $\gamma$ makes a jump to $\beta$.
\end{lem}

\begin{proof} Since there is a jump from $\alpha$ to $\beta$, it follows that $a_1 < b_1$ and $a_2 < b_2$.

    If $a<g<b$, then $m - a_1 + a_2 < m - g_1 + g_2 < m - b_1 + b_2$, so \[a_2 - a_1 < g_2 - g_1 < b_2 - b_1.\]
        
        Suppose $g_1 < b_1$. Then $g_2 - g_1 < b_2 - b_1 < b_2 - g_1$ which implies $g_2 < b_2$. Thus, $\gamma$ makes a jump to $\beta$.
        
        On the other hand, suppose $b_1 \leq g_1$.  In this case $a_1 < b_1 \leq g_1$, so $a_2 - g_1 < a_2 - a_1 < g_2 - g_1$. This yields $a_2 < g_2$. which means $\alpha$ makes a jump to $\gamma$.
        
        The case where $b<g<a$ leads to a similar argument and the desired conclusion.\end{proof}

\begin{lem}\label{lem:outside S}
    Suppose there is a jump from $\alpha=(a_1,a_2) \in D_a$ to $\beta=(b_1,b_2) \in D_b$ and define $\ell = \min\{b_1-a_1,b_2-a_2\}$. 
    \begin{enumerate}
        \item[\rm{a.)}] If $a<b$ and $\gamma \in D_g$ such that $a-\ell < g < a$ or $b < g < b+\ell$, then either there is a jump from $\alpha$ to $\gamma$ or there is a jump from $\gamma$ to $\beta$. 
        
        \item[\rm{b.)}] If $b<a$ and $\gamma \in D_g$ such that $b-\ell < g < b$ or $a < g < a+\ell$, then either there is a jump from $\alpha$ to $\gamma$ or there is a jump from $\gamma$ to $\beta$.
    \end{enumerate}
\end{lem}
\begin{proof}
    Let $a<b$ and $\gamma = (g_1,g_2) \in D_g$ and first consider $a - \ell < g < a$.  If $a_2 < g_2$, then $a_2 - g_1 < g_2 - g_1 < a_2 - a_1$ implying $a_1 < g_1$ and $\alpha$ makes a jump to $\gamma$.  Alternatively, suppose $g_2 \leq a_2$. Since $a_2<b_2$, it follows that $g_2<b_2$. Note that $a-\ell<g<a$ implies $a_2-a_1-\ell < g_2-g_1 < a_2 - a_1$. This yields $g_1<a_1+\ell$. Since $\ell\leq b_1-a_1$, it follows that $g_1< b_1$. Thus, $\gamma$ makes a jump to $\beta$.
    
    Second, consider $b < g < b + \ell$. If $g_2 < b_2$, then $g_2 - b_1 < b_2 - b_1 < g_2 - g_1$ implying $g_1 < b_1$, and there is a jump from $\gamma$ to $\beta$.  Alternatively, suppose $b_2 \leq g_2$. Since $a_2<b_2$, it follows that $a_2<g_2$. Then $b<g<b+\ell$ implies $b_2 - b_1 < g_2-g_1 < b_2 - b_1 + \ell$. This yields $b_1-\ell < g_1$. Since $\ell\leq b_1-a_1$, it follows that $a_1 < g_1$. Thus, $\alpha$ makes a jump to $\gamma$.
    
    The result when $b<a$ follows from a similar argument.
\end{proof}

Lemma \ref{lem:diagsum} says that given the diagonal of two elements the diagonal of the sum and difference of the elements lands in.  For this reason, the authors refer to this as the `Landing Lemma.' This lemma becomes important quickly, as many succeeding proofs require a working knowledge of the locations of contributing diagonals, non-contributing diagonals, or elements of a jump.

\begin{lem}[Landing Lemma]\label{lem:diagsum}
 If $\alpha \in D_a$, $\beta \in D_b$, and $\alpha + \beta \in [m]\times[n]$, then $\alpha + \beta \in D_{a + b - m}$.  Similarly, if $\alpha-\beta \in \mn$ then $\alpha-\beta\in D_{a-b+m}.$
\end{lem}

\bpf
Let $\alpha = (a_1, a_2)$ and $\beta = (b_1, b_2)$ and suppose $\alpha+\beta\in D_k$ for some appropriately defined $k$. Then, $m-a=a_1-a_2$ and $m-b=b_1-b_2$. Moreover, $m-k = a_1+b_1-(a_2+b_2)$ and thus 
\[ k = m-(a_1+b_1)+a_2+b_2 = a+b-m. \] A similar analysis gives that $\alpha -\beta \in D_{a-b+m}$.
\epf
  
         Let $c$ be a coloring of $\mn$. Recall, $s_2 = \min\{x ~|~ c((x,x)) \neq c((1,1))\}$. Define $W_1 = \{(x,y) \in \mn ~|~ (x,y) + (s_2, s_2) \in \mn \}$, $W_2 = \{(x,y) \in \mn ~|~ (x,y) - (s_2, s_2) \in \mn\}$, and $W = W_1 \cup W_2$.  See Figure \ref{fig: W-Blocks} for a visualization of these, and the following, definitions.
         
         The next definitions will not be used until Section \ref{sec:L} but are included here as they are closely related to $W$, $W_1$ and $W_2$. Define  $Y_1 = \{(x,y) \in \mn : m < x +s_2 \text{ and } y-s_2 < 0\}$, $Y_2 = \{(x,y) \in \mn : x-s_2 < 0 \text{ and } m < y +s_2  \}$, and $Y = Y_1 \cup Y_2$. 
         
         Let $c$ be a coloring of $\mn$.  If $\alpha \in D_a, \beta \in D_{a+1}$ where $D_a$ and $D_{a+1}$ are contributing, and $c(\alpha),c(\beta)\notin c(D_m)$, then $\{\alpha, \beta\}$ are a \emph{consecutive contributing pair of elements}. If $\{\alpha,\beta\}$ is a consecutive contributing pair of elements with $\alpha = (a_1,a_2)$ and $\beta = (a_1, a_2+1)$, then $\{\alpha, \beta\}$ is a \emph{horizontal pair} and if $\alpha = (a_1,a_2)$ and $\beta = (a_1-1, a_2)$, then $\{\alpha, \beta\}$ is a \emph{vertical pair}.  Let $P_v$ be a vertical pair, $P_h$ be a horizontal pair, and define $P = P_v\cup P_h$. If $P_v \cap P_h = \emptyset$, $P_v\cap W \neq \emptyset$, and $P_h \cap W \neq \emptyset$, $P$ is called a \emph{contributing disjoint corner}.

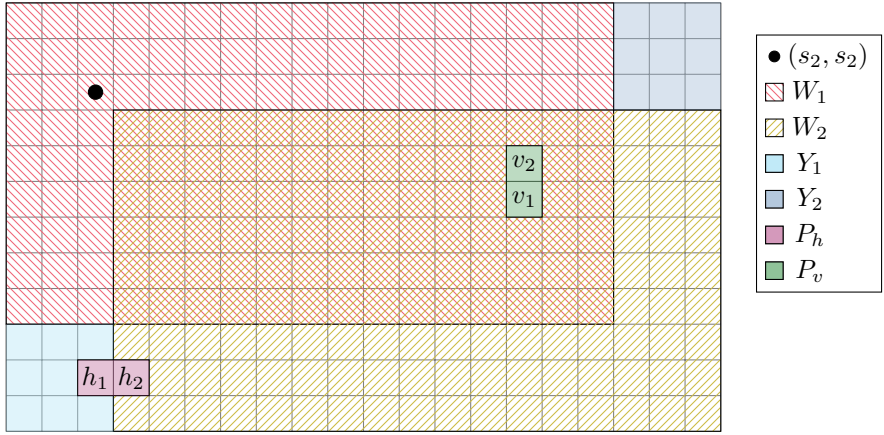
\begin{figure}[ht!]
    \centering
    
    \begin{tikzpicture}[scale = .95]
        \draw[step=.5cm,gray,thin] (0,0) grid (10,6);
        
        \filldraw[fill=white, draw=black, opacity = .2] (0,0) rectangle (1.5,1.5);
        \filldraw[fill=white, draw=black, opacity = .2] (8.5,4.5) rectangle (10,6);
        
        \filldraw[pattern = South east lines, pattern color = cbred] (0,6) rectangle (8.5,1.5);
        \filldraw[pattern = South west lines, pattern color = cbyellow] (1.5,4.5) rectangle (10,0);
        
        \filldraw[color=black] (1.25,4.75) circle (.1);
        
        \filldraw[fill=cbcyan, draw=black, opacity = .2] (0,0) rectangle (1.5,1.5);
        \filldraw[fill=cbblue, draw=black, opacity = .2] (8.5,4.5) rectangle (10,6);
        
        \draw[fill=white] (1,.5) rectangle (2,1);
        \filldraw[draw = black, fill=cbpurple, opacity = .3] (1,.5) rectangle (1.5,1);
        \filldraw[draw = black, fill=cbpurple, opacity = .3] (1.5,.5) rectangle (2,1);
        \node at (1.25,.75) {$h_1$};
        \node at (1.75,.75) {$h_2$};
        
        \draw[fill=white] (7,3) rectangle (7.5,4);
        \filldraw[draw = black, fill=cbgreen, opacity = .3] (7,3) rectangle (7.5,3.5);
        \filldraw[draw = black, fill=cbgreen, opacity = .3] (7,3.5) rectangle (7.5,4);
        
        \node at (7.25,3.25) {$v_1$};
        \node at (7.25,3.75) {$v_2$};
        
        \draw (10.5,1.95) rectangle (12.25,5.55);

        \draw[fill=black] (10.75,5.25) circle (.075);
        \node at (11.5,5.25) {$(s_2,s_2)$}; 
        \node[draw, pattern = South east lines, pattern color = cbred] at (10.75,4.75) {};
        \node at (11.25,4.75) {$W_1$};
        \node[draw, pattern = South west lines, pattern color = cbyellow] at (10.75,4.25) {};
        \node at (11.25,4.25) {$W_2$};
        \node[draw, fill = cbcyan, opacity = .4] at (10.75,3.75) {};
        \node[draw] at (10.75,3.75) {};
        \node at (11.25,3.75) {$Y_1$};
        \node[draw, fill = cbblue, opacity = .4] at (10.75,3.25) {};
        \node[draw] at (10.75,3.25) {};
        \node at (11.25,3.25) {$Y_2$};
        \node[draw, fill = cbpurple, opacity = .5] at (10.75,2.75) {};
        \node[draw] at (10.75,2.75) {};
        \node at (11.25,2.75) {$P_h$};
        \node[draw, fill = cbgreen, opacity = .5] at (10.75,2.25) {};
        \node[draw] at (10.75,2.25) {};
        \node at (11.25,2.25) {$P_v$};
        
    \end{tikzpicture}
    
    \caption{An example of $\mn$ when $(s_2, s_2) = (3,3)$ with corresponding $Y_1$, $Y_2$, $W_1$ and $W_2$ highlighted. Further, a contributing disjoint corner, $P_h \cup P_v$, with $c(h_1),c(h_2),c(v_1)$ and $c(v_2)$ pairwise distinct, is shown.
    }
    
    \label{fig: W-Blocks}
\end{figure}

        Lemma \ref{lem: hor-vert-pairs} says that a contributing disjoint corner will force a rainbow solution in $\mn$. The proof of Theorem \ref{thm:main3colors} uses Lemma \ref{lem: hor-vert-pairs} to arrive at a contradiction once a contributing disjoint corner is found.
        
\begin{lem}\label{lem: hor-vert-pairs}
   If $c$ is a rainbow-free coloring of $\mn$ for $eq$ with $m \le n$, then there are no contributing disjoint corners $P = P_v \cup P_h$ in $\mn$.
\end{lem}

\begin{proof}
    If $|c(D_m)| = 1$, then $(s_2,s_2)$ does not exist and $W$ is not defined so the statement is vacuously true.  Therefore, let $2 \le |c(D_m)|$. 
    For the sake of contradiction, assume $P_v = \{(a_1, b_1), (a_1 - 1, b_1)\}$ and $P_h = \{(a_2, b_2), (a_2, b_2 + 1)\}$ form a contributing disjoint corner. Define $c_3 = c((a_1, b_1))$, $c_4 = c((a_1 - 1, b_1))$, $c_5 = c((a_2, b_2))$, and $c_6 = c((a_2, b_2 + 1))$, and recall that convention implies $c((1,1))=1$ and $c((s_2,s_2))=2$. Note that $1,2,c_3,c_4,c_5$ and $c_6$ are pairwise distinct.  It will be shown that every possible color of $(s_2-1,s_2)$ gives a rainbow solution, which is a contradiction.
    \par
    
    First, note that $P_v$ either intersects $W_1$ or $W_2$. Suppose $P_v \cap W_1 \neq \emptyset$. Then $(a_1-1,b_1) \in W_1$ implying that the elements of \[(a_1-1,b_1) + (s_2,s_2) = (a_1-1+s_2,b_1+s_2)\] are in $\mn$. Further, this equation implies that $c((a_1-1+s_2,b_1+s_2)) \in \{2,c_4\}$. So, the equation \[(a_1,b_1) + (s_2-1,s_2) = (a_1-1+s_2,b_1+s_2)\] implies that $c((s_2-1,s_2)) \in \{2,c_3,c_4\}$. Alternatively, if $P_v \cap W_2 \neq \emptyset$, then $(a_1,b_1) \in W_2$. Similar to the previous case, the equations \[(a_1-s_2,b_1-s_2) + (s_2,s_2) = (a_1,b_1)\]  and  \[(a_1-s_2,b_1-s_2) + (s_2-1,s_2) = (a_1-1,b_1)\] imply that $c((s_2-1,s_2)) \in \{2,c_3,c_4\}$.
    \par
    Second, note that $P_h$ either intersects $W_1$ or $W_2$. If $P_h \cap W_1 \neq \emptyset$, then $(a_2,b_2) \in W_1$.
    So, $(a_2+s_2,b_2+s_2) \in \mn$. Since $a_2 < a_2+s_2-1 < a_2+s_2$, it follows that $(a_2+s_2-1,b_2+s_2) \in \mn$. Thus, the equation \[(a_2,b_2+1) + (s_2-1,s_2-1) = (a_2+s_2-1,b_2+s_2)\] implies $c((a_2+s_2-1,b_2+s_2)) \in \{1,c_6\}$ because $c((s_2-1,s_2-1))=1$. Finally, the equation \[(a_2,b_2) + (s_2-1,s_2) = (a_2+s_2-1,b_2+s_2)\] implies $c((s_2-1,s_2)) \in \{1,c_5,c_6\}$. On the other hand, if $P_h \cap W_2 \neq \emptyset$, then $(a_2,b_2+1)\in W_2$ and the equations \[(a_2-s_2+1,b_2-s_2+1)+(s_2-1,s_2)=(a_2,b_2+1)\]and \[(a_2-s_2+1,b_2-s_2+1)+(s_2-1,s_2-1)=(a_2,b_2)\] imply that $c((s_2-1,s_2)) \in \{1,c_5,c_6\}$.
    \par
    Therefore, no matter where $P_v$ and $P_h$ are in $\mn$, $c((s_2-1,s_2)) \in \{2,c_3,c_4\} \cap \{1,c_5,c_6\}$, a contradiction.
\end{proof}


\section{Main diagonal has three colors}\label{sec:maindiagthree}

This section focuses on the situation where $\mn$ has an exact $(m+n+1)$-coloring and the main diagonal has exactly three colors.  
The overall strategy of the section is to analyze the structure of an arbitrary coloring, with the previously stated assumptions, and show that the structural restrictions must admit a rainbow solution to $eq$.

Lemma \ref{lem:eachdiagcontributes} is a direct consequence of Lemmas \ref{onedistinctcolorlem} and \ref{lem:ellminusthree}.

\begin{lem}\label{lem:eachdiagcontributes}
   If $c$ is an exact, rainbow-free $(m+n+1)$-coloring of $\mn$ for $eq$ with $3\le m \le n$ and $|c(D_m)| = 3$, then each off-diagonal $D_k$ contributes exactly one color $c_k$ such that $c_k \notin c(\mn \setminus D_k)$.
\end{lem}

Lemma \ref{lem:3colornojumps} proves that there are no jumps between distinctly colored elements when there are three colors in the main diagonal. A brief outline of the argument follows.  First, it is shown that if such a jump exists between elements in $ D_a$ and $ D_b$, then there is a $k$ such that $a = m+k$, and $b = m+2k$. Then, it is shown that $|k| = 1$, so either $a = m+1$ and $b = m+2$ or $a= m-1$ and $b = m-2$. This ultimately forces $\alpha = (1,2)$ or $\alpha = (2,1)$, respectively.   Finally, two cases pertaining to the size of $s_3$ lead to contradictions. 

\begin{lem}\label{lem:3colornojumps}
    If $c$ is an exact, rainbow-free $(m+n+1)$-coloring of $\mn$ for $eq$ with $3\le m \le n$ and $|c(D_m)| = 3$, then there are no jumps from $\alpha$ to $\beta$ with $c(\alpha),c(\beta)\notin c(D_m)$ and $c(\alpha) \neq c(\beta)$.
\end{lem}

\begin{proof}
    Suppose there is a jump from $\alpha \in D_a$ to $\beta \in D_b$ with $c(\alpha),c(\beta) \notin c(D_m)$ and $c(\alpha) \neq c(\beta)$. There must exist some $\delta \in D_t$ with $\alpha + \delta = \beta$, and $c(\delta) \in \{c(\alpha),c(\beta)\}$. If $t \notin \{a,b,m\}$, then Lemma \ref{lem:eachdiagcontributes} is contradicted.  Further, $t\neq m$ because $c(\delta) \notin c(D_m)$, and $t \neq b$ because $\alpha \notin D_m$. Thus, $t = a$, $c(\delta) = c(\alpha)$, and Lemma \ref{lem:diagsum} implies $b = 2a-m$. 
    
    Define $k = a-m$ so that $a = m+k$ and $b = m+2k$. For the sake of contradiction, assume $1<|k|$. Then there exist some off-diagonal $D_g$ with $g$ strictly between $a$ and $b$. By Lemma \ref{lem:eachdiagcontributes}, $D_g$ must contribute a color, say $D_g$ contributes $c(\gamma)$ for some $\gamma\in D_g$. Lemma \ref{lem:inside S} implies that there must be either a jump from $\alpha$ to $\gamma$ or a jump from $\gamma$ to $\beta$. In either case, there exists some $\delta' \in D_{t'}$ such that $\alpha + \delta' = \gamma$ or $\gamma + \delta' = \beta$, and Lemma \ref{lem:diagsum} implies $t' \in \{m+g-a,m+b-g\}$.  
    For the sake of contradiction, assume $t' \in \{a,b,m,g\}$ implying $t' \in \{a,b,m,g\} \cap \{m+g-a,m+b-g\}$. Since $g$ is strictly between $a$ and $b$, $g=m+k+j$ for some $j$ strictly between $0$ and $k$. Converting all elements in the two sets to be in terms of $m$, $k$ and $j$ gives
    \[
        t' \in \{m+k,m+2k,m,m+k+j\}\cap \{m+j,m+k-j\}. 
    \]
      Note that the restrictions on $j$, and the fact that $1<|k|$, imply $m+j$ and $m+k-j$ are strictly between $m$ and $m+k$.  Thus, the intersection is empty and $t'\notin\{a,b,m,g\}$. Since $c$ is rainbow-free, $c(\delta') \in \{c(\alpha),c(\beta),c(\gamma)\}$.  This contradicts Lemma \ref{lem:eachdiagcontributes}.
    
    This means $|k| = 1$, which implies that $a = m+1$ and $b = m+2$ or $a = m-1$ and $b = m-2$. First, consider the case where $a = m+1$ and $b = m+2$. Assume for the sake of contradiction that $\alpha = (a_1,a_2), \beta = (b_1,b_2), \delta = (t_1,t_2)$ are defined such that $2 \leq a_1,t_1$. Since $2 \leq a_1,t_1$, it must be the case that $4 \leq b_1$. Additionally, note that $D_{m+3}$ contributes $c(\gamma')$ for some $\gamma' = (g_1,g_2)\in D_{m+3}$ by Lemma \ref{lem:eachdiagcontributes}. Since $\alpha \in D_{m+1}$, $\beta \in D_{m+2}$, $\delta \in D_{m+1}$, $\text{ and } \gamma' \in D_{m+3}$, the equations:
    
    \begin{gather}
        \label{eqa2}a_2 = a_1 + 1\\
        \label{eqt2}t_2 = t_1 + 1\\
        \label{eqg2}g_2 = g_1 + 3
    \end{gather}
    
    \noindent hold. Using the above information, it will be shown that irrespective of the location of $\gamma' = (g_1, g_2)$ in the diagonal $D_{m+3}$ a contradiction can be found.
    
    \begin{description}
        \item[Case 1.] Suppose $a_1 < g_1$.\\
        Equations (\ref{eqa2}) and (\ref{eqg2}) give $a_2 = a_1 + 1 < a_1 + 3 < g_1 + 3 = g_2$. Since $a_1 < g_1$ and $a_2 < g_2$ there is a jump from $\alpha$ to $\gamma'$, so by Lemma \ref{lem:diagsum} there exists a $\delta' \in D_{m+2}$ such that $\alpha + \delta' = \gamma'$. Since $c$ is a rainbow-free coloring, $c(\delta') \in \{c(\alpha), c(\gamma')\}$. However, $D_{m+2}$ contributes the color $c(\beta)$ so Lemma \ref{onedistinctcolorlem} is contradicted.
        
        \item[Case 2.] Suppose $g_1 \leq a_1$.\\ Since there is a jump from $\alpha$ to $\beta$, $a_1< b_1$ and $a_2 < b_2$. Thus, $g_1 < b_1$. Equations (\ref{eqa2}) and (\ref{eqg2}) give $g_2 \leq a_2 + 2$. Further, Equation (\ref{eqt2}) implies $t_1 \leq 3$, so $\alpha + \delta = \beta$ gives $a_2 + 3 \leq a_2 + t_2 = b_2$.  So $g_2 < b_2$.  This implies there is a jump from $\gamma'$ to $\beta$, so by Lemma \ref{lem:diagsum} there exists a $\delta' \in D_{m-1}$ such that $\gamma' + \delta' = \beta$.  Lemma \ref{lem:eachdiagcontributes} implies $D_{m-1}$ must contribute a color.  Thus, $c(\delta') \not\in \{c(\gamma'), c(\beta)\}$, contradicting that $c$ is rainbow-free.
    \end{description}
    
    In either case, a contradiction is obtained, so $a_1 = 1$ or $t_1 = 1$, that is, $\alpha=(1,2)$ or $\delta=(1,2)$.

    Since $c(\alpha)=c(\delta)$, without loss of generality, suppose $\alpha=(1,2)$. Now, assume, for the sake of contradiction, that $\frac{m}{2}+1 \leq s_3$. Define $\rho = (s_3-1,s_3-2)$ so that $\alpha + \rho = (s_3,s_3)$, and $c(\rho) \in \{c(\alpha),3\}$. Since $\rho \in D_{m-1}$, Lemmas \ref{onedistinctcolorlem} and \ref{lem:eachdiagcontributes} imply that $c(\rho)=3$. Lemma \ref{lem:eachdiagcontributes} further implies that $D_{m-1}$ must contribute a color, say $D_{m-1}$ contributes $c(\xi)$ for some $\xi = (e_1,e_2) \in D_{m-1}$. If $e_1 < s_3-1$, then there exists some $\kappa = (k,k) \in D_m$ such that $k < s_3$ and $\xi + \kappa = \rho$, a rainbow solution. Otherwise, suppose $s_3-1 < e_1$. Then there exists some $\kappa = (k,k) \in D_m$ such that $\rho + \kappa = \xi$. Recall $e_1 \leq m$ and $\frac{m}{2} + 1 \le s_3$, so
    \[k = e_1 - (s_3 - 1) \leq m - s_3 + 1 \leq (2s_3 - 2) - s_3 + 1  < s_3.\] Thus, $\rho + \kappa = \xi$ is a rainbow solution. Both cases contradict that $c$ is rainbow-free.
    
    Now consider when $s_3 < \frac{m}{2}+1$.  Lemma \ref{FGRWWrainbowlem} implies $4 \leq s_3$ so $7 \leq m \leq n$. Note that if $\chi = (x_1,x_2) \in D_{m+3}$ is such that $2 \leq x_1$, then there exists a jump from $\alpha$ to $\chi$ implying that $c(\chi) \in c(D_m)$.  Since $D_{m+3}$ must contribute a color and the only way it can contribute is if $x_1 = 1$, $D_{m+3}$ must contribute $c((1,4))$. If $5 \leq b_2$ (and $3 \leq b_1$), then there is a jump from $(1,4)$ to $\beta$. So, there exists some $\delta'' \in D_{m-1}$ with $c(\delta'') \in \{c((1,4)),c(\beta)\}$, which is a contradiction by Lemmas \ref{onedistinctcolorlem} and \ref{lem:eachdiagcontributes}. Thus $b_2 \in \{3,4\}$ but if $b_2 = 3$ there is not a jump from $\alpha$ to $\beta$ so $\beta = (2,4)$. Finally, $\alpha + \beta = (3,6)$. So, $c(3,6) \in \{c(\alpha),c(\beta)\}$, but $(3,6) \in D_{m+3}$, which is a contradiction by Lemmas \ref{onedistinctcolorlem} and \ref{lem:eachdiagcontributes} .  Therefore, contradictions are obtained in every situation when $a = m+1$ and $b=m+2$.
    
    The case when $a = m-1$ and $b = m-2$ is very similar. Essentially, reflect the coordinates and diagonals from the argument above across the main diagonal.  The analysis forces $\alpha = (2,1)$ or $\delta = (2,1)$ and $\rho = (s_3-2,s_3-1)$ is considered.  However, the cases at this point differ slightly.  Instead of scrutinizing $\frac{m}{2}+1 \leq s_3$ and $s_3 < \frac{m}{2}+1$ the cases that are considered are $\frac{m+1}{2}+1 \leq s_3$ and $s_3 < \frac{m+1}{2}+1$.  In any event, contradictions are established in all situations so there are no jumps from $\alpha$ to $\beta$ with $c(\alpha),c(\beta) \notin c(D_m)$ and $c(\alpha) \neq c(\beta)$. \end{proof}

Lemma \ref{lem:s3block} bounds $s_3$ which in turn bounds $s_2$ in the proof of Theorem \ref{thm:main3colors}.

\begin{lem}\label{lem:s3block}
Suppose $c$ is a exact, rainbow-free $(m+n+1)$-coloring of $[m]\times[n]$ for $eq$ with $3\le m \le n$ and $|c(D_m)|=3$. If $i,j<s_3$, then $c((i,j)) \in c(D_m)$.
\end{lem}

\bpf
If $(i,j) \in D_m$, then $c((i,j)) \in c(D_m)$. So, suppose $(i,j) \in D_k$, with $k\neq m$, and assume, for the sake of contradiction, that $c((i,j)) \notin c(D_m)$. By Lemma \ref{lem:eachdiagcontributes}, $D_k$ contributes $c((i,j))$. Define $\alpha \in D_a$ so that $(i,j) + \alpha = (s_3,s_3)$ so $c(\alpha)\in \{c((i,j)), 3$\}. Lemma \ref{lem:diagsum} implies $\alpha$ and $(i,j)$ are in different diagonals, and Lemma \ref{lem:eachdiagcontributes} gives $c(\alpha) \neq c((i,j))$. Thus, $c(\alpha) = 3$. 
\par
Let $\beta = (b_1,b_2) \in D_a$ be distinct from $\alpha$. It will be shown $c(\beta) \in c(D_m)$. First, suppose there exists a $\gamma = (g_1,g_2) \in D_m$ such that $\alpha + \gamma = \beta$. If $g_1 < s_3$, then $c(\gamma)\in\{1,2\}$. So, $c(\gamma) \neq c(\alpha)$ implying $c(\beta) \in c(D_m)$. If $s_3 \leq g_1$, then $s_3 \leq g_2$ implying $i < s_3 \leq g_1 < b_1$ and $j < s_3 \leq g_2 < b_2$. So there is a jump from $(i,j)$ to $\beta$, and Lemma \ref{lem:3colornojumps} implies $c(\beta) \in c(D_m)$. 
\par
Otherwise, suppose $\beta + \gamma = \alpha$. Then $g_1 < s_3$, so $c(\gamma)\in \{1,2\}$ which implies $c(\alpha) \neq c(\gamma)$ and $c(\beta) \in c(D_m)$. Thus, $c(D_{a})\setminus c(D_m) = \emptyset$, contradicting Lemma \ref{lem:eachdiagcontributes}. Therefore, $c((i,j)) \in c(D_m)$.
\epf

Lemma \ref{thm:main3colors} gives that there are no rainbow-free $(m+n+1)$-colorings of $m\times n$ with no more than $3$ colors in the main diagonal. 
Assuming that there is such a coloring, this proof shows that there must exist a vertical and horizontal pair that intersect $W$. 
This gives a contributing disjoint corner, which contradicts Lemma \ref{lem: hor-vert-pairs}.

\begin{thm}\label{thm:main3colors}
    If $c$ is an exact $(m+n+1)$-coloring of $\mn$ with $3\le m \le n$ and $|c(D_m)|\le 3$, then $\mn$ contains a rainbow solution to $eq$.  
\end{thm}

\begin{proof}
    Assume, for the sake of contradiction, that $c$ is rainbow-free. Then Corollary \ref{mainlessthan3} implies $|c(D_m)|=3$. Note Lemma \ref{lem:eachdiagcontributes} implies every off-diagonal is a contributing diagonal.  In particular, $D_{m-1}$ and $D_{m-2}$ contribute a color.  Suppose for some $\alpha = (a_1,a_1-1) \in D_{m-1}$, $D_{m-1}$ contributes $c(\alpha)$ and for some $\beta = (b_1,b_1-2) \in D_{m-2}$, $D_{m-2}$ contributes $c(\beta)$. Thus, $P = \{\alpha,\beta\}$ is a consecutive contributing pair of elements and Lemma \ref{lem:3colornojumps} indicates $P$ is a vertical or horizontal pair.  Further, Lemma \ref{lem:s3block} implies that $s_3\leq a_1,b_1$. Note that $2 \leq s_2 \leq \frac{s_3}{2} \leq \frac{a_1}{2}$ with the middle inequality being a result of Lemma \ref{FGRWWrainbowlem}. Continuing, \[s_2 \leq \frac{a_1}{2} = a_1 - \frac{a_1}{2} < a_1-1 < a_1.\]
     So, $\alpha = (a_1,a_1-1) \in W_2$ and $P$ intersects $W$.   A similar argument indicates a vertical or horizontal pair, call it $P'$, exists and intersects $D_{m+1}$, $D_{m+2}$ and $W_2$. 
    \par
    Assume, for the sake of contradiction, that $P$ is a vertical pair. Then $a_1-1 = b_1-2$. Define $K$ to be the minimum $k$ such that $D_k$ contributes the color of an element of the $a_1-1$ column. More precisely, 
    
    \[K = \min\{ k : D_k \text{ contributes } c((x,a_1-1)) \text{ for some $x$} \}.\] Notice that $K$ exists and $K \le m-2$.  By construction, $D_K$ contributes $c((x,a_1-1))$ for $x=m-K+a_1-1$. 
   
    By Lemmas \ref{lem:eachdiagcontributes} and \ref{lem:3colornojumps}, $D_{K-1}$ must contribute $c((x+1,a_1-1))$ or $c((x,a_1-2))$. By the minimality of $K$, $D_{K-1}$ must contribute $c((x,a_1-2))$. This implies $P_1 =\{(x,a_1-2),(x,a_1-1)\}$ is a horizontal pair. Now, since $s_2 < a_1-1 < m-K+a_1-1 = x$, so $P_1$ intersects $W_2$.  However, now $P_1$ and $P'$ or $P$ and $P'$ are contributing disjoint corners which contradicts Lemma \ref{lem: hor-vert-pairs}.  Thus, $P$ must be a horizontal pair.
    
    Note that $D_{m+1}$ must contribute $c(\gamma)$ for some $\gamma = (g_1,g_1+1) \in D_{m+1}$. Then $s_3 \leq g_1+1$ by Lemma \ref{lem:s3block}. Now, define $L$ to be the maximum $\ell$ such that $D_\ell$ contributes the color of an element of the $g_1$ row. Specifically, \[L = \max\{\ell : D_\ell \text{ contributes } c((g_1,y)) \text{ for some $y$}\}.\] Again, notice that $m+1 \le L$.  So, $D_L$ contributes $c((g_1,y))$ for $y = L-m+g_1$.

    Then Lemma \ref{lem:3colornojumps} and the maximality of $L$ implies that $D_{L+1}$ must contribute $c((g_1-1,y))$. Then $P_2=\{(g_1,y),(g_1-1,y)\}$ is a vertical pair that intersects $W_2$, using a similar argument as before. Finally, since $s_2 \geq 2$, Lemma \ref{FGRWWrainbowlem} implies $s_2 < s_3-1 \leq g_1 < L-m+g_1 = y$. So, $(g_1,y) \in W_2$ implying that $P_2$ and $P$ form a contributing disjoint corner, which contradicts Lemma \ref{lem: hor-vert-pairs}.  Therefore, $\mn$ contains a rainbow solution.
\end{proof}

\section{Four or More Colors in Main Diagonal}\label{sec:nojumps}

The main result of Section \ref{sec:maindiagthree} states that every exact $(m+n+1)$-coloring of $\mn$ will contain a rainbow solution to $eq$ if there are three colors in the main diagonal. Section \ref{sec:nojumps} culminates with Theorem \ref{thm:nojumpsever} which shows that if an exact $(m+n+1)$-coloring of $\mn$ contains a jump, then a rainbow solution to $eq$ exists.  To start, a useful set of diagonals is defined.
   
   For each $\delta\in \mn$ define the set of off-diagonals \[\mathbb{D}_{\delta} := \{D_g \neq D_m \,:\, \text{for all $\gamma\in D_g$, either $\gamma+\delta$ or $\gamma-\delta$ is in $\mn$} \}.\]

 Lemma \ref{alternatinglem} gives an upper bound for the number of colors that can exist in $[m]\times [n]$ that do not appear in the main diagonal. The upper bound is written with respect to the cardinality of some $\mathbb{D}_{\delta}$ with $c(\delta) \notin c(D_m)$. The proof proceeds with two cases, when there are at least $\frac{|\mathbb{D}_{\delta}|}{3}$ non-contributing diagonals in $\mathbb{D}_{\delta}$ and when there are less than $\frac{|\mathbb{D}_{\delta}|}{3}$ non-contributing diagonals in $\mathbb{D}_{\delta}$. The former case is immediate. In the latter case, it is shown that there are at least $\frac{2|\mathbb{D}_{\delta}|}{3}-3$ contributing diagonals that are each associated with some non-contributing diagonal and at most two of these contributing diagonals are associated with the same non-contributing diagonal. This leads to the conclusion that there are at least $\frac{|\mathbb{D}_{\delta}|}{3}-\frac{3}{2}$ non-contributing diagonals in $[m]\times [n]$. Subtracting this from the total number of off-diagonals gives the upper bound.

\begin{lem}\label{alternatinglem} Let $c$ be an exact, rainbow-free $(m+n+1)$-coloring of $[m]\times[n]$ for $eq$ with $3 \le m\leq n$. For each $\delta$ such that $c(\delta)\notin c(D_m)$,  \[|c([m]\times[n])\setminus c(D_m)|\leq m+n-\frac{1}{2}-\frac{|\mathbb{D}_\delta|}{3}.\]
\end{lem}
    
\begin{proof}
    Assume $\delta \in D_t$ and suppose there are $k$ non-contributing diagonals in $\DD_\delta$, and assume
     $k \geq \frac{|\DD_\delta|}{3}$.  Since there are $m+n-2$ off-diagonals and at least $\frac{|\DD_\delta|}{3}$ are non-contributing, it follows that 
     
    \[\left|c([m]\times[n])\setminus c(D_m)\right|\leq m+n-2-\frac{|\mathbb{D}_\delta|}{3} < m+n-\frac{1}{2}-\frac{|\mathbb{D}_\delta|}{3}.\]
    
    On the other hand, suppose $k < \frac{|\DD_\delta|}{3}$. This implies there are at least $\frac{2|\DD_\delta|}{3}$ contributing diagonals in $\DD_\delta$. Note that at most one of these contributing diagonals contains $c(\delta)$. Define $D_x$ to be the diagonal in $\mn$ that contributes $c(\delta)$.  This means there are at least $\frac{2|\DD_\delta|}{3} - 1$ contributing diagonals in $\DD_\delta$ that do not contain $c(\delta)$. Let $D_g$ be one of those diagonals. Then there exists $\gamma \in D_g$ such that $D_g$ contributes $c(\gamma)$ with $c(\gamma)\notin c(D_m)\cup \{c(\delta)\}$.  Since $D_g \in \DD_\delta$, either $\gamma+\delta \in \mn$ or $\gamma-\delta \in \mn$ implying either $\{\gamma,\delta,\gamma+\alpha\}$ or $\{\delta, \gamma-\delta, \gamma\}$ is a solution in $\mn$.  Specifically, $\gamma + \delta\in D_{g+t-m}$ and $\gamma-\delta \in D_{g-t+m}$ by Lemma \ref{lem:diagsum}.
    
    Since $c$ is rainbow-free, either $c(\gamma+\delta)$ or $c(\gamma-\delta)$ is in $\{c(\gamma),c(\delta)\}$.  This implies either $c(D_{g+t-m})$ or $c(D_{g-t+m})$ contains $c(\gamma)$ or $c(\delta)$. If $D_{g+t-m}$ or $D_{g-t+m}$ contains $c(\gamma)$, then, since $D_g$ contributes $c(\gamma)$, either $D_{g+t-m}$ or $D_{g-t+m}$ is non-contributing. If $D_{g+t-m}$ or $D_{g-t+m}$ contains $c(\delta)$, then $D_{g+t-m}$ or $D_{g-t+m}$ is contributing if and only if $D_{g+t-m} = D_x$ or $D_{g-t+m} = D_x$.  This analysis shows that if $D_g\in \{D_{x+t-m},D_{x-t+m}\}$, then $D_g$ does not necessarily correspond to a non-contributing diagonal.  Applying similar analysis to $D_x$ indicates that if $D_x\in \DD_\delta$, then $D_x$ does not necessarily correspond to a non-contributing diagonal.
    
    Define $G$ to be the set of all contributing diagonals in $\DD_{\delta}\setminus\{D_x,D_{x+t-m},D_{x-t+m}\}$.  Thus, each $D_g \in G$ corresponds to at least one non-contributing off-diagonal, namely $D_{g+t-m}$ or $D_{g-t+m}$, and
    
    \[ \frac{2|\DD_\delta|}{3}-3\le |G|.\]
    
    If $G$ is empty, then $|\DD_\delta| \le 9/2$ so
    
    \begin{align*}
        \left|c([m]\times[n])\setminus c(D_m)\right| & \leq m+n-2 \\ 
        &  \le
        m+n-\frac{1}{2}-\frac{|\DD_\delta|}{3}.
    \end{align*}
    
    So assume $G$ is non-empty and let $D_g \in G$ so that $D_{g+t-m}$ or $D_{g-t+m}$ is non-contributing. Without loss of generality, suppose $D_{g+t-m}$ is non-contributing. For the sake of contradiction, suppose $D_{h},D_{h'} \in G\setminus \{D_g\}$ are distinct with $D_{g+t-m} \in \{D_{h+t-m},D_{h-t+m}\}\cap\{D_{h'+t-m},D_{h'-t+m}\}$. Since $g \notin \{h,h'\}$, it must be that $g+t-m=h-t+m$ and $g+t-m=h'-t+m$ implying $h = h'$, a contradiction. So, there is at most one diagonal $D_h$ in $G$ such that $D_{g+t-m} \in \{D_{h+t-m},D_{h-t+m}\}$.  In other words, at most two diagonals in $G$ correspond to the same non-contributing off-diagonal.

    \par
    Since $D_g \in G$ was arbitrary, it can be concluded that at least $\frac{|G|}{2}$ diagonals are non-contributing. Thus, 
    
    \begin{align*}
        \left|c([m]\times[n])\setminus c(D_m)\right| & \leq m+n-2-\frac{|G|}{2} \\ 
        & \leq m+n-2-\frac{\frac{2|\DD_\delta|}{3}-3}{2} \\
        &  = m+n-\frac{1}{2}-\frac{|\DD_\delta|}{3}.
    \end{align*}
\end{proof}

As stated, the main goal of this section is to show that a jump with distinctly colored elements yields a rainbow solution.
Propositions \ref{upperleftlem} and \ref{upperboundjumps} give lower and upper bounds, respectively, on the jump distance which leaves only a few cases to analyze.

\begin{prop}\label{upperleftlem} Let $c$ be an exact, rainbow free $(m+n+1)$-coloring of $[m]\times[n]$ for $eq$ with $3\le m\leq n$. If there exists an $\delta = (d_1,d_2)$ such that $c(\delta)\notin c(D_m)$, then \[\frac{4m+9-6\left\lfloor\log_2(m)+1\right\rfloor}{4} \le d_1+d_2.\]  
 \end{prop}

\begin{proof}
    Define \[{S}_{\delta}:= \{\gamma \in \mn\, |\,  \gamma+\delta,\gamma-\delta \notin \mn\}.\] Notice $\delta' = (m-d_1+1,d_2)\in D_{d_1+d_2-1}$ and $\delta'\pm\delta \notin \mn$ (see Figure  \ref{fig:D_{delta}}). Since there are $m+n-2$ off-diagonals and $S_\delta$ intersects $2(d_1+d_2-1)$ diagonals, it follows $|\DD_{\delta}| = m+n-2 - 2(d_1+d_2-1) = m+n-2d_1-2d_2$. If $S_\delta$ intersects $D_m$, then $|\DD_{\delta}| = m+n-2d_1-2d_2+1$. So, $|\DD_{\delta}| \geq m+n-2d_1-2d_2.$
    Therefore, Lemma \ref{alternatinglem} implies that 
    
    \begin{equation}\label{ineq1}
    \begin{split}
        |c([m]\times[n])\setminus c(D_m)| & \leq m+n-\frac{1}{2} - \frac{|\DD_\delta|}{3} \\
        & \leq m+n-\frac{1}{2}-\frac{m+n-2d_1-2d_2}{3}\\
       & = \frac{4m+4n-3+4d_1+4d_2}{6}.
    \end{split}
    \end{equation}
    
    Because $c$ is a rainbow-free, Theorem \ref{FGRWWrainbowthm} implies $|c(D_m)|\leq \left\lfloor\log_2(m)+1\right\rfloor$.  Further, $c$ is an exact $m+n+1$-coloring, so
    \begin{equation}\label{ineq2}
        m+n+1-\left\lfloor\log_2(m)+1\right\rfloor\leq |c([m]\times[n])\setminus c(D_m)|.
    \end{equation}
    Inequalities (\ref{ineq1}) and (\ref{ineq2}) give
    \[m+n+1-\left\lfloor\log_2(m)+1\right\rfloor\leq |c([m]\times[n])\setminus c(D_m)|\leq \frac{4m+4n-3+4d_1+4d_2}{6}. \]
    Isolating $d_1 + d_2$ and using $m\le n$ gives
    
    \begin{equation}\label{ineq3} \frac{4m+9-6\left\lfloor\log_2(m)+1\right\rfloor}{4}\le \frac{2m+2n+9-6\left\lfloor\log_2(m)+1\right\rfloor}{4} \le d_1+d_2.
    \end{equation}\end{proof}
    
    \begin{figure}[ht!]
        \centering  
        \begin{tikzpicture}[scale = .75] 
            \draw[step=.5cm,gray,thin] (0,0) grid (10,6);
            \filldraw[fill=cbred, draw=black, opacity = .3] (0,0) rectangle (1.5,2);
            \filldraw[fill=cbred, draw=black, opacity = .3] (8.5,4) rectangle (10,6);
            \node at (.75,1) {$S_{\delta}$};
            \node at (9.25,5) {$S_{\delta}$};
            
            \node at (1.25,4.25) {$\delta$};
            
            \node at (1.25,1.75) {$\delta'$};
            
            \filldraw[fill=cbcyan, draw = black, opacity=.2] (3,0) rectangle (10,.5);
            \filldraw[fill=cbcyan, draw = black, opacity=.2] (2.5,.5) rectangle (10,1);
            \filldraw[fill=cbcyan, draw = black, opacity=.2] (2,1) rectangle (10,1.5);
            \filldraw[fill=cbcyan, draw = black, opacity=.2] (1.5,1.5) rectangle (10,2);
            \filldraw[fill=cbcyan, draw = black, opacity=.2] (1,2) rectangle (10,2.5);
            \filldraw[fill=cbcyan, draw = black, opacity=.2] (0.5,2.5) rectangle (10,3);
            \filldraw[fill=cbcyan, draw = black, opacity=.2] (0,3) rectangle (9.5,3.5);
            \filldraw[fill=cbcyan, draw = black, opacity=.2] (0,3.5) rectangle (9,4);
            \filldraw[fill=cbcyan, draw = black, opacity=.2] (0,4) rectangle (8.5,4.5);
            \filldraw[fill=cbcyan, draw = black, opacity=.2] (0,4.5) rectangle (8,5);
            \filldraw[fill=cbcyan, draw = black, opacity=.2] (0,5) rectangle (7.5,5.5);
            \filldraw[fill=cbcyan, draw = black, opacity=.2] (0,5.5) rectangle (7,6);
            
            \foreach \n in {0,...,11}
            {
                \filldraw[fill=white, draw = black!55]
                (.5*\n,6-.5*\n) rectangle (.5*\n+.5,5.5-.5*\n);
            }
            
            \node at (5,3) {\Large $\DD_{\delta}$};
        \end{tikzpicture}
        \caption{A visualization of $\delta$, $\delta'$ and $S_\delta$ (in red) and $\DD_\delta$ (in cyan) from Proposition \ref{upperleftlem}. }\label{fig:D_{delta}}
    \end{figure}
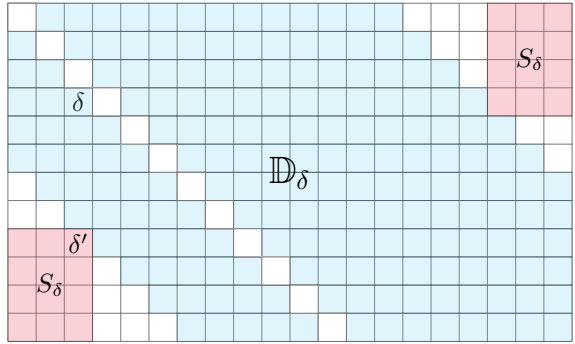

Lemmas \ref{jump with alpha}, \ref{jump with beta} and \ref{Glem} support Proposition \ref{upperboundjumps}.
The assumptions for all of these results are that there is a jump from $\alpha = (a_1,a_2)$ to $\beta = (b_1,b_2)$, $\alpha \in D_a$ and $\beta\in  D_b$, $c(\alpha),c(\beta) \notin c(D_m)$ and $c(\alpha)\neq c(\beta)$.
Also, $\alpha + (d_1,d_2) = \beta$.  The goal is to find a relationship between the jump distance, $d_1 + d_2$, and the number of contributing diagonals.  
The main challenge arises during the analysis of the diagonals that surround $D_a$ and $D_b$.
Technically, this set of surrounding diagonals is $S$ which was defined by (\ref{eqS}) and visualized in Figure \ref{fig:jump def}.  During the analysis of $S$ many iterations of jumps need to be considered.  For example, if $\gamma$ is an element in one of the diagonals of $S$ then there must be a jump from $\alpha$ to $\gamma$ (or from $\gamma$ to $\beta$) which introduces another element $\zeta$ where $\alpha + \zeta = \gamma$ (or $\gamma + \zeta = \beta$).  Lemma \ref{jump with alpha} starts investigating where $\zeta$ could be in the $\alpha + \zeta = \gamma$ case and Lemma \ref{jump with beta} starts investigating where $\zeta$ could be in the $\gamma + \zeta = \beta$ case.  Lemma \ref{Glem} adds more clarification about where the $\zeta$ elements can be.

\begin{lem}\label{jump with alpha}
    Let $c$ be an exact, rainbow free $(m+n+1)$-coloring of $[m]\times[n]$ for $eq$ with $3\le m \le n$, and suppose there is a jump from $\alpha\in D_a$ to $\beta\in D_b$ such that $\alpha+\delta=\beta$ for some $\delta\in D_t$. If $c(\alpha),c(\beta)\notin c(D_m)$ and $c(\alpha)\neq c(\beta)$, then there is at most one diagonal $D_g$ that contains an element $\gamma$ with $c(\gamma)\notin c(D_m\cup \{\alpha,\beta\})$ such that $\alpha+\zeta=\gamma$ for some $\zeta\in D_z$ with $z \in\{a,b,t,m,g\}$.
    Moreover, if such a $D_g$ exists, then $z=a$.
\end{lem}

\begin{proof}
    Assume there exist diagonals $D_g$ and $D_{g'}$ such that for some $\gamma\in D_g$ and $\gamma'\in D_{g'}$, $c(\gamma),c(\gamma')\notin c(D_m\cup \{\alpha,\beta\})$ and there is a $\zeta\in D_z$ and $\zeta'\in D_{z'}$ such that $\alpha+\zeta = \gamma$ and $\alpha+\zeta'=\gamma'$ with $z\in\{a,b,t,m,g\}$ and $z'\in\{a,b,t,m,g'\}$.
    
    Since $a\neq m$, it follows that $z\neq g$ and $z'\neq g'$, so $z,z' \in \{a,b,t,m\}$. Additionally,  $c(\gamma)\neq c(\alpha)$ and $ c(\gamma')\neq c(\alpha)$ imply that $c(\zeta)\in \{c(\alpha),c(\gamma)\}$ and $c(\zeta')\in \{c(\alpha),c(\gamma')\}$.  Thus, $z,z' \notin \{m,b\}$.  Since $c(\gamma), c(\gamma') \notin c(D_m \cup \{\alpha, \beta\})$, Lemma \ref{onedistinctcolorlem} implies $g \neq b$ and $g' \neq b$. Thus, by Lemma \ref{lem:diagsum}, $z\neq t$ and $z' \neq t$. So, it must be the case that $z=z'=a$.  Finally, Lemma \ref{lem:diagsum} gives that $g=g'=2a-m$, that is $D_g$ is unique.
\end{proof}

\begin{lem}\label{jump with beta}
    Let $c$ be an exact, rainbow free $(m+n+1)$-coloring of $[m]\times[n]$ for $eq$ with $3\le m \le n$, and suppose there is a jump from $\alpha\in D_a$ to $\beta\in D_b$ such that $\alpha+\delta=\beta$ for some $\delta\in D_t$. If $c(\alpha),c(\beta)\notin c(D_m)$ and $c(\alpha)\neq c(\beta)$, then there is at most one diagonal $D_g$ that contains an element $\gamma$ with $c(\gamma)\notin c(D_m\cup \{\alpha,\beta\})$ such that $\gamma+\zeta=\beta$ for some $\zeta\in D_z$ with $z \in\{a,b,t,m,g\}$.
    Moreover, if such a $D_g$ exists, then $z=g$.
\end{lem}

\begin{proof}
    Assume there exist diagonals $D_g$ and $D_{g'}$ such that for some $\gamma\in D_g$ and $\gamma'\in D_{g'}$, $c(\gamma),c(\gamma')\notin c(D_m\cup \{\alpha,\beta\})$ and there is a $\zeta\in D_z$ and $\zeta'\in D_{z'}$ such that $\gamma+\zeta = \beta$ and $\gamma'+\zeta'=\beta$ with $z\in\{a,b,t,m,g\}$ and $z'\in\{a,b,t,m,g'\}$.

    Since $c(\gamma)\neq c(\beta)$ and $c(\gamma')\neq c(\beta)$, it follows that $c(\zeta)\in \{c(\beta),c(\gamma)\}$ and $c(\zeta')\in \{c(\beta),c(\gamma')\}$.  Thus, $z,z' \notin \{a,m\}$.  Also, $g\neq m$ and $g'\neq m$ imply that $z\neq b$ and $z'\neq b$.  Since $c(\gamma), c(\gamma') \notin c(D_m \cup \{\alpha, \beta\})$, Lemma \ref{onedistinctcolorlem} implies $g\neq a$ and $g' \neq a$. Thus, by Lemma \ref{lem:diagsum}, $z \neq t$ and $z'\neq t$.  Therefore, $z=g$ and $z'= g'$.  Now, by Lemma \ref{lem:diagsum}, $b = z+g-m = 2g-m$ and $b = z'+g'-m = 2g'-m$. So, $g = g'$, showing $g$ is unique.
\end{proof}

\begin{lem}\label{Glem}
    Let $c$ be an exact, rainbow free $(m+n+1)$-coloring of $[m]\times[n]$ for $eq$ with $3\le m \le n$ and suppose there is an element $\alpha\in D_a$ and an element $\beta\in D_b$ such that $c(\alpha)\neq c(\beta)$, and $c(\alpha),c(\beta)\notin c(D_m)$. Let
    \[S' \subseteq \{D_x ~|~ c(D_x)\backslash c(D_m\cup D_a \cup D_b) \neq\emptyset\} \] and \[G := \{ D_{x}\in S'~|~ \text{if $c(D_x)\backslash c(D_m) = c(D_y)\backslash c(D_m)$ for some $D_y\in S'$, then $x \le y$} \}.\] If $D_{g_i}\in G$ and $\gamma_i\in D_{g_i}$ with $c(\gamma_i)\notin c(D_m\cup\{\alpha,\beta\})$ and there is a corresponding $\zeta_i\in D_{z_i}$ such that either $\alpha+\zeta_i=\gamma_i$ or $\gamma_i+\zeta_i =\beta$, then $D_{z_i}\notin G\setminus \{D_{g_i}\}$. Moreover, if $i\neq j$, then $z_i\neq z_j$. 
\end{lem}

\begin{proof}
      For the sake of contradiction, suppose $D_{z_i}\in G\setminus \{D_{g_i}\}$. By definition of $G$ there is a $\rho\in D_{z_i}$ such that $c(\rho)\notin c(D_m\cup \{\alpha,\beta\})$. Since $z_i\neq g_i$, it follows that $c(\rho)\neq c(\gamma_i)$.  Further, $\alpha+\zeta_i=\gamma_i$ or $\gamma_i+\zeta_i = \beta$ implies  $c(\zeta_i)\in c(\{\alpha,\gamma_i,\beta\})$. Thus, $c(\zeta_i)\neq c(\rho)$, and $c(\zeta_i),c(\rho)\notin c(D_m)$, which contradicts Lemma \ref{onedistinctcolorlem}. Therefore, $D_{zi}\notin G\setminus\{D_{gi}\}$.
      
      Let $i\neq j$, $\gamma_j\in D_{g_j}$ with $c(\gamma_j) \notin c(D_m\cup \{\alpha,\beta\})$ have corresponding $\zeta_j\in D_{z_j}$.  If $\alpha +\zeta_i = \gamma_i$ and $\alpha +\zeta_j = \gamma_j$, or $\gamma_i + \zeta_i = \beta$ and $\gamma_j + \zeta_j = \beta$, then $g_i\neq g_j$ and Lemma \ref{lem:diagsum} imply $z_i\neq z_j$.  On the other hand, without loss of generality, if $\alpha +\zeta_i = \gamma_i$ and $\gamma_j +\zeta_j = \beta$, then $c(\zeta_i)\in \{c(\gamma_i),c(\alpha)\}$ and $c(\zeta_j)\in \{c(\gamma_j),c(\beta)\}$.  This implies $c(\zeta_i)\neq c(\zeta_j)$ so Lemma \ref{onedistinctcolorlem} implies $z_i\neq z_j$.
\end{proof}

\begin{prop}\label{upperboundjumps}
Let $c$ be an exact, rainbow-free $(m+n+1)$-coloring of $[m]\times[n]$ for $eq$ with $3\le m \le n$. If $\alpha+(d_1,d_2)= \beta$ corresponds to a jump from $\alpha$ to $\beta$ where $c(\alpha),c(\beta)\notin c(D_m)$ and $c(\alpha)\neq c(\beta)$, then $d_1+d_2\leq 2\log_2(m)+1$.  
\end{prop}

\begin{proof} Let $\alpha=(a_1,a_2) \in D_a$, $\beta = (b_1,b_2)\in D_b$, and  $\delta = (d_1,d_2)\in D_t$.  Since $c$ is rainbow-free, $c(\delta)\in  \{c(\alpha),c(\beta)\}$. Define $S$ as in (\ref{eqS}).
    By Lemmas \ref{lem:outside S} and \ref{lem:inside S}, every element of every diagonal in $S$ has a solution with $\alpha$ or $\beta$.  In particular, if $\gamma$ is such an element, then there is either a jump from $\alpha$ to $\gamma$ or from $\gamma$ to $\beta$. 
    
     Further, define \[S':= \{D_x\in S ~|~ c(D_x)\backslash c(D_m\cup D_a \cup D_b)\neq \emptyset\},\] 
     
     and
   
    \[G := \{ D_{x}\in S'~|~ \text{if $c(D_x)\backslash c(D_m) = c(D_y)\backslash c(D_m)$ for some $D_y\in S'$, then $x \le y$} \}.\]

    Also, define \[T:= \{D_x\in [m]\times[n]~ |~ D_x\notin S\cup \{D_m,D_a,D_b\}\},\] and 
    \[T' := \left\{D_x\in T~|~ c(D_x)\not\subseteq c\left(D_m\cup D_a\cup D_b\cup \bigcup_{D_s \in S} D_s\right)\right\}.\]
    Accounting for the facts that $D_a,D_b,D_m\notin S$ and $\delta=\beta-\alpha$, \[|S|\geq (m+b_2-a_1)-(m+a_2-b_1)-1-3=d_1+d_2-4.\]
   
   Suppose $|G|=k$ and reindex $G=\{D_{g_1},\ldots, D_{g_k}\}$.  Note that for each $D_{g_i}\in G$ there exists $\gamma_i\in D_{g_i}$ with $c(\gamma_i)\notin c(D_m\cup \{\alpha,\beta\})$ and $c(\gamma_i) \neq c(\gamma_j)$ when $i\neq j$. 
   In addition, as $D_{g_i} \in S$, there exists a diagonal $D_{z_i}$ with $\zeta_i \in D_{z_i}$ such that $\alpha+\zeta_i=\gamma_i$  or $\gamma_i +\zeta_i= \beta$. 
   Define $Z := \{D_{z_i} ~|~ \zeta_i \in D_{z_i} \text{ for } 1\le i\le k\}$.  By Lemmas \ref{jump with alpha} and \ref{jump with beta}, at least $k-2$ diagonals $D_{z_i}$ of $Z$ are not in $\{D_a,D_b,D_t,D_m,D_{g_i}\}$. Define \[Z' := \begin{cases}
        \big\{D_{z_i} \in Z ~|~ D_{z_i} \notin \{D_a,D_b,D_t,D_m,D_{g_i}\}\big\} & \text{if $a=t$,} \\
        \big\{D_{z_i} \in Z ~|~ D_{z_i} \notin \{D_a,D_b,D_t,D_m,D_{g_i}\}\big\} \cup \{D_t\} & \text{if $a\neq t$.}
   \end{cases}\]
  
   For all $D_{z_i},D_{z_j}\in Z\cap Z'$, Lemma \ref{Glem} implies that $D_{z_i} \notin G$ and $z_i \neq z_j$ when $i\neq j$.  It follows that \[Z' = \begin{cases}
        \big\{D_{z_i} \in Z ~|~ D_{z_i} \notin \{D_a,D_b,D_t,D_m\} \cup G\big\} & \text{if $a=t$,} \\
        \big\{D_{z_i} \in Z ~|~ D_{z_i} \notin \{D_a,D_b,D_t,D_m\} \cup G\big\} \cup \{D_t\} & \text{if $a\neq t$.}
   \end{cases}\] Lemmas \ref{jump with alpha} and \ref{jump with beta} imply $k-1\leq|Z'|$.

    Note that $c(\zeta_i) \in \{c(\alpha),c(\beta),c(\gamma_i)\}$, so $c(\zeta_i)\notin c(D_m)$ and $|c(D_{z_i})\setminus c(D_m\cup D_a\cup D_b\cup D_{g_i})|=0$ by Lemma \ref{onedistinctcolorlem}. 
    Also, $c(\delta) \in \{c(\alpha),c(\beta)\}$, so $|c(D_t)\setminus c(D_m\cup D_a\cup D_b|=0$. 
    Thus, \[c\left(\bigcup_{D_x\in Z'}D_{x}\right)\subseteq c\left( D_m\cup D_a\cup D_b\cup \bigcup_{D_{g_i}\in G}D_{g_i} \right).\] Note that by definition, for all $D_{x}\in Z'$, $D_{x}\notin G\cup \{D_a, D_b,D_m\}$, so $Z'\subseteq(T\setminus T')\cup (S\setminus G)$.

     If $k\leq \frac{|S|}{2}$, then there are at least $\frac{|S|}{2}$ off-diagonals in $S\setminus G$. Therefore, there are at least $\frac{|S|}{2}$ off-diagonals in $(T\setminus T')\cup (S\setminus G)$.  If $k\ge\frac{|S|+1}{2}$, then $|Z'|\geq k-1\ge \frac{|S|-1}{2}$. Since $Z'\subseteq (T\setminus T')\cup (S\setminus G)$, there are at least $\frac{|S|-1}{2}$ off-diagonals in $(T\setminus T')\cup (S\setminus G)$.

      Recall, $|S| \geq  d_1+d_2-4$.  This implies there are at least \[\frac{d_1+d_2-5}{2}\] off-diagonals in $(T\setminus T')\cup (S\setminus G)$.  
    
     Because there are a total of $m+n-2$ off-diagonals and Lemma \ref{onedistinctcolorlem} implies that each off-diagonal can contain at most one color not in the main diagonal, there are at most  \[m+n-2-\frac{d_1+d_2-5}{2}\] colors in $c(\mn)\backslash c(D_m)$.  Therefore, there are at least \[m+n+1-\left(m+n-2-\frac{d_1+d_2-5}{2}\right) = 3+\frac{d_1+d_2-5}{2}\] colors in $c(D_m)$.
     By Theorem \ref{FGRWWrainbowthm}, \[|c(D_m)| \leq \log_2(m)+1,\] so \[3+\frac{d_1+d_2-5}{2}\leq \log_2(m)+1.\] Therefore, \[d_1+d_2\leq 2\log_2(m)+1.\] \end{proof}
   
   Combining the inequalities in Propositions \ref{upperleftlem} and \ref{upperboundjumps} gives the following Corollary.

\begin{cor}\label{m<10}
    Let $c$ be an exact, rainbow-free $(m+n+1)$-coloring of $[m]\times[n]$ for $eq$ with $m\leq n$. If $11\le m$ or $14\le n$, then there are no jumps between elements $\alpha$ and $\beta$ such that $c(\alpha),c(\beta)\notin c(D_m)$ and $c(\alpha)\neq c(\beta)$. Furthermore, when $8\leq m \leq 10$, it follows that $5\leq d_1 +  d_2 \leq 7$.
\end{cor}

\bpf Assume there is a jump from $\alpha$ to $\beta$.  In particular, $\alpha + (d_1,d_2) = \beta$.  Then, the inequalities from Propositions \ref{upperleftlem} and \ref{upperboundjumps} must be satisfied, so

\begin{equation}\label{bound}\frac{4m+9-6\left\lfloor\log_2(m)+1\right\rfloor}{4} \le d_1 + d_2 \le 2\log_2(m)+1.\end{equation}

This can only happen if $m \le 10$.  Further, Inequality (\ref{ineq3}) implies $n\le 13$ and if $8\le m \le 10$, Inequality (\ref{bound}) implies $5 \le d_1 + d_2 \le 7$.
\epf

Lemma \ref{lem:2m-t=b} is used to generalize Corollary \ref{m<10} into Theorem \ref{thm:nojumpsever}.  The proof of Lemma \ref{lem:2m-t=b} uses analysis similar to the proof of Proposition \ref{upperboundjumps} but the additional constraint that $|c(D_m)| = 4$, deduced using Corollary \ref{m<10},  leads to deeper case analysis. 

\begin{lem}\label{lem:2m-t=b}
    Let $c$ be an exact, rainbow free $(m+n+1)$-coloring of $[m]\times[n]$ for $eq$ with $3\le m \le n$, and suppose there is a jump from $\alpha\in D_a$ to $\beta\in D_b$ such that $\alpha+\delta=\beta$ for some $\delta = (d_1,d_2) \in D_t$. If $c(\alpha),c(\beta)\notin c(D_m)$ and $c(\alpha)\neq c(\beta)$, then $2m-t=b$.
\end{lem}

\begin{proof}
    Corollary \ref{mainlessthan3} and Lemma \ref{lem:3colornojumps} imply $4\le |c(D_m)|$. Lemma \ref{lem:s2} and Corollary \ref{m<10} further restrict to the situation to $8 \le m \le 10$, $n \le 13$, $5 \le d_1 + d_2 \le 7$, and $|c(D_m)|=4$. For the sake of contradiction, assume $2m-t \neq b$. First, note that if $2m-t=a$, then subtracting $m-t$ from both sides gives $m = a+t-m$ and $a+t-m=b$ by Lemma \ref{lem:diagsum} implying $m=b$, a contradiction. Thus, $2m-t \notin \{a,b\}$.
    
    Notice $d_1 + d_2 \leq 7$ and $8 \leq s_4$ imply that $\varsigma = (s_4,s_4)-\delta \in \mn$ and more specifically, $\varsigma \in D_{2m-t}$ by Lemma \ref{lem:diagsum}. It will be shown that one of $D_a,D_b,D_{2m-t}$ is non-contributing. Note that $\delta + \varsigma = (s_4,s_4)$ implies $c(\varsigma) \in \{c(\delta),4\} \subset \{c(\alpha),c(\beta),4\}$. If $c(\varsigma) \in \{c(\alpha),c(\beta)\}$, then at least one of $D_a,D_b,D_{2m-t}$ is non-contributing. Otherwise, suppose $c(\varsigma)=4$. It will be shown that $D_{2m-t}$ is non-contributing, the desired result. So, let $\rho = (p_1,p_2) \in D_{2m-t}$. If $p_1<s_4-d_1$, then there exists some $k < s_4$ such that $\rho + (k,k) = \varsigma$ implying $c(\rho) \in c(D_m)$. Second, suppose $s_4-d_1<p_1<2s_4-d_1$. Then there exists some $k < s_4$ such that $(k,k) + \varsigma = \rho$ implying $c(\rho) \in c(D_m)$. Finally, suppose $2s_4-d_1 \leq p_1$ which means $2s_4-d_2 \leq p_2$. Notice that $p_1 \leq m \leq 10$ and that $8 \leq s_4$ by Corollary \ref{cor:2power}. So, $d_1 \geq 2s_4-p_1 \geq 16-10 \geq 6$. But then $d_1+d_2\leq 7$ implies $d_2 \leq 1$. So, $p_2 \geq 2s_4-d_2 \geq 16-1 = 15 > n$. Thus, no such $\rho$ exists proving that $c(D_{2m-t}) \subseteq c(D_m)$ implying that $D_{2m-t}$ is non-contributing.
    
    Thus, every diagonal $D_k$ with $k \notin \{m,a,b,2m-t\}$ must contribute a color distinct from $c(\alpha)$ and $c(\beta)$. Consider $D_t$. Since $\alpha + \delta = \beta$, it follows that $c(\delta) \in \{c(\alpha),c(\beta)\}$ and Lemma \ref{onedistinctcolorlem} implies $c(D_t) \subseteq c(D_m)\cup\{c(\alpha),c(\beta)\}$. So, $D_t$ can't contribute a color distinct from $c(\alpha)$ and $c(\beta)$. Thus, $t \in \{m,a,b,2m-t\}$. However, $a\neq b$ implies $t\neq m$ which implies $t\neq 2m-t$. Further $a\neq m$ implies $t\neq b$. So, $t=a$. 
    
    Define $S$ as in (\ref{eqS}) and note that \[|S| = 
    \begin{cases}
        d_1+d_2-3 & \text{if $m \notin  \{D_x \,|\, m+a_2-b_1 < x < m+b_2-a_1\}$}, \\
        d_1+d_2-4 & \text{if $m \in  \{D_x \,|\, m+a_2-b_1 < x < m+b_2-a_1\}$}.
    \end{cases}\] 
    
    Since $5 \leq d_1+d_2$, it follows that $1\leq |S|$. It is claimed that if $D_{2m-t} \notin S$, then there exists some $D_g\in S\backslash \{D_{2m-t}\}$.  So assume $D_{2m-t}\in S$. By definition of $S$, $m+a_2-b_1 < 2m-t < m+a_1 - b_2$ which implies $m+a_2-b_1 < a+b - (2m-t) < m+a_1 - b_2$.  Thus, $D_{a+b-(2m-t)}\in S\backslash\{ D_{2m-t}\}$ because if $a+b-(2m-t) \in \{a,b,m,2m-t\}$, it can be concluded that either $b = 2m-t$ or $a=m$, both of which are contradictions.  Thus, there exists some $D_g\in S$ such that $g\neq 2m-t$.

    Since $g \notin \{m,a,b,2m-t\}$, it follows that $D_g$ must contribute $c(\gamma)$ for some $\gamma \in D_g$. By Lemmas \ref{lem:inside S} and \ref{lem:outside S}, there exists a $\zeta \in D_z$ such that $\alpha + \zeta = \gamma$ or $\gamma + \zeta = \beta$. In either case, $c(\zeta) \in \{c(\alpha), c(\beta), c(\gamma)\}$ implying that if $z \notin \{m,a,b,2m-t,g\}$, then $z$ is non-contributing, a contradiction. Lemmas \ref{jump with alpha} and \ref{jump with beta} imply that 
    \begin{itemize}
        \item $z=a$ and $\alpha + \zeta = \gamma$, 
        \item $z=g$ and $\gamma + \zeta = \beta$, or
        \item $z=2m-t$.
    \end{itemize}
    If $z=a$, then $z=t$ and Lemma \ref{lem:diagsum} implies $g = a+z-m = a+t-m = b$, contradicting that $D_g \in S$. If $z=g$, then Lemma \ref{lem:diagsum} implies $b=z+g-m=2g-m$. Additionally, since $t=a$, it follows that $b=t+a-m=2a-m$. So, $g=a$, again contradicting that $D_g \in S$. Thus, $z=2m-t$. If $\alpha + \zeta = \gamma$, then Lemma \ref{lem:diagsum} implies $g = a+z-m = a+(2m-t)-m = a+(2m-a)-m = m$, contradicting that $D_g \in S$. So, $\gamma + \zeta = \beta$. This means $c(\zeta) \in \{c(\beta),c(\gamma)\}$. And since $c(D_{2m-t}) \subseteq c(D_m) \cup \{c(\alpha),c(\beta)\}$, it follows that $c(\zeta)=c(\beta)$. Recall that $c(\beta) \in c(D_{2m-t})$ only if $c(\delta) = c(\beta)$. Since $\delta \in D_a$, $c(\alpha),c(\beta)\in c(D_a)$, contradicting Lemma \ref{onedistinctcolorlem}. Thus, it has finally been shown that $b=2m-t$.
\end{proof}

Theorem \ref{thm:nojumpsever} says that there are no jumps between distinctly colored elements whose colors do not appear in the main diagonal. First it is shown that $|D_m| = 4$ and one of $D_t,D_a,D_b$ is non-contributing, so by Lemma \ref{lem:ellminusthree}, all other off-diagonals are contributing. Next, some $\gamma\in D_g$ and $\zeta \in D_z$ are found such that $\alpha + \zeta = \gamma$ or $\gamma+\zeta = \beta$. Since $D_z$ must be contributing, $z\in \{g,a,b,t\}$, and Lemma \ref{jump with alpha} and \ref{jump with beta} imply that $\alpha+\zeta = \gamma$ and $z = a$ or $\gamma + \zeta = \beta$ and $z = g$. In either case, a contradiction is found.

\begin{thm}\label{thm:nojumpsever}
    Let $c$ be an exact, rainbow free $(m+n+1)$-coloring of $[m]\times[n]$ for $eq$ with $3\le m \le n$, and suppose there is a jump from $\alpha\in D_a$ to $\beta\in D_b$ such that $\alpha+\delta=\beta$ for some $\delta = (d_1,d_2) \in D_t$. Then $c(\alpha)\in c(D_m)$ or $c(\beta)\in c(D_m)$ or $c(\alpha) = c(\beta)$.
\end{thm}

\begin{proof}
    For the sake of contradiction, assume $c(\alpha),c(\beta)\notin c(D_m)$ and $c(\alpha) \neq c(\beta)$. Corollary \ref{mainlessthan3} and Lemma \ref{lem:3colornojumps} imply $4\le |c(D_m)|$. Lemma \ref{lem:s2} and Corollary \ref{m<10} further restrict to the situation to $8 \le m \le 10$, $n \le 13$, $5 \le d_1 + d_2 \le 7$, and $|c(D_m)|=4$. Additionally, Lemma \ref{lem:2m-t=b} implies $2m-t = b$. 
    
    Since $a=b$, $a=t$, or $b=t$ implies $m \in \{t,a,b\}$, it follows that $a$, $b$, and $t$ are pairwise distinct. Specifically, $t<m<b<a$ or $a<b<m<t$. The solution $\alpha + \delta = \beta$ means $c(\delta) \in \{c(\alpha),c(\beta)\}$. Now, Lemma \ref{lem:ellminusthree} implies that two of $D_t,D_a,D_b$ contribute $c(\alpha)$ and $c(\beta)$ and the third is non-contributing.
    Thus, every diagonal $D_k$ with $k \notin \{m,t,a,b\}$ must contribute a color distinct from $c(\alpha)$ and $c(\beta)$.
    
    Define $S$ as in (\ref{eqS}). Note that \[|S| = 
    \begin{cases}
        d_1+d_2-3 & \text{if $m \notin  \{D_x \,|\, m+a_2-b_1 < x < m+b_2-a_1\}$}, \\
        d_1+d_2-4 & \text{if $m \in  \{D_x \,|\, m+a_2-b_1 < x < m+b_2-a_1\}$}
    \end{cases}\]
    
   and observe $1 \leq |S|$. It is claimed that $1 \leq |S\setminus \{D_t\}|$. Indeed, if $D_t \in S$, then $|S| \geq 3$ since if $t < m < b$, it follows that
    \[
        m+a_2-b_1 < t \leq b-2 = m-b_1+b_2-2     \]
    so that
    \[2 < b_2-a_2 = d_2 < d_1\] 
    which implies $d_1+d_2\geq 7$, thus $d_1+d_2 = 7$. A similar case holds for $b < m < t$.
    
    Let $D_g \in S \setminus \{D_t\}$, and note that $D_g$ must contribute a color. Say $D_g$ contributes $c(\gamma)$ for some $\gamma \in D_g$. By Lemmas \ref{lem:inside S} and \ref{lem:outside S}, there exists a $\zeta \in D_z$ such that $\alpha + \zeta = \gamma$ or $\gamma + \zeta = \beta$. Note that $c(\zeta) \in \{c(\gamma),c(\alpha),c(\beta)\}$, 
    so $z \in \{g,a,b,t\}$. By Lemmas \ref{jump with alpha} and \ref{jump with beta}, this can only happen if $\alpha + \zeta = \gamma$ and $z=a$, or $\gamma + \zeta = \beta$ and $z=g$. 
    \begin{description}
        \item[Case 1.] Suppose $\alpha + \zeta = \gamma$ and $z=a$.\\ 
        Lemma \ref{lem:diagsum} implies \[g = 2a-m = m-2a_1+2a_2.\] Since $D_g \in S$, it follows that $g < m+b_2-a_1$ implying \[a_2-a_1<b_2-a_2=d_2\]
        and similarly the lower bound implies $a_1-a_2<d_1$. In the case when $t<m<a$, then $d_2<d_1$ and $a_1<a_2$. Since $d_1+d_2=7$, $0< a_2-a_1 < 3$. Similarly, if $a<m<t$, then $0< a_1-a_2 < 3$. Since $b$ is between $m$ and $a$,
        this would require $a=m+2$ or $a=m-2$. Then $b=m+1$ and $g=m+4$ or $b=m-1$ and $g=m-4$, respectively.
        
        First, suppose $a=m+2$. Since $D_{m+4} \in S$, it follows that $D_{g'} := D_{m+3} \in S$ and so there exists a $\gamma'\in D_{g'}$ such that $D_{m+3}$ contributes $c(\gamma')$. By Lemmas \ref{lem:inside S} and \ref{lem:outside S}, there exists a jump $\zeta' \in D_{z'}$ from $\alpha$ to $\gamma'$ or from $\gamma'$ to $\beta$, with $c(\zeta')\in \{c(\gamma'),c(\alpha),c(\beta)\}$. By Lemmas \ref{jump with alpha} and \ref{jump with beta}, this can only happen if $\zeta'$ is a jump from $\alpha$ to $\gamma'$ and $z'=a$, or $\zeta'$ is a jump from $\gamma'$ to $\beta$ and $z'=g'$. If $z'=a=m+2$, then $g'=m+4$ since $\zeta'$ is a jump from $\alpha$ to $\gamma'\in D_{g'}$. Likewise, if $z'=g'=m+3$ then $b=m+6$. In either case, a contradiction arises.

        Similar contradictions can be found when $a=m-2$.
        
        \item[Case 2.] Suppose $\gamma + \zeta = \beta$ and $z=g$. 
        
        Then $b = 2g-m$. Additionally, since $t=2m-b$ and $b=a+t-m$, it follows that $a=2b-m$. Combining these equations yields $a=4g-3m$ implying 
        \begin{equation}\label{eqg} g=\frac{3m+a}{4}=m+\frac{a_2-a_1}{4}.\end{equation}
        In parallel to Case 1, since $D_g \in S$, it follows that $g < m+b_2-a_1$ implying $a_2-a_1<4b_2-4a_1$ and \[3(a_1-a_2)<4(b_2-a_2)=4d_2.\]
        
        Similarly, the lower bound on $S$ implies $3(a_2-a_1)<4d_1$. In the case when $a<m<t$, then $d_1 < d_2$ and $a_2<a_1$.  
        Since $d_1+d_2=7$, it follows that $0< a_1-a_2 < \frac{4}{3}d_2 \leq 8$. Likewise, if $t<m<a$, then $0< a_2-a_1 < 8$. 
        Note that Equation (\ref{eqg}) implies $a_1-a_2 = 4(m-g)$, i.e. $4$ divides $a_1-a_2$. 
        So, $a=m+4$ or $a=m-4$. 
        In the former case,  $D_{m+3} \in S$ with $m+3\notin \{g,t\}$, and in the latter, $D_{m-3} \in S$ with $m-3\notin \{g,t\}$. 
        Lemmas \ref{lem:inside S},  \ref{lem:outside S},\ref{jump with alpha} and \ref{jump with beta} yield similar contradictions to those in Case 1. 
        \end{description}
        
        Therefore, $c(\alpha)\in c(D_m)$ or $c(\beta)\in c(D_m)$ or $c(\alpha) = c(\beta)$.  In other words, there are no jumps between distinctly colored elements.
\end{proof}

\section{Rainbow number of $\mn$ for $x_1+x_2 = x_3$}\label{sec:L}

This final section uses definitions that were introduced in Section \ref{sec:genlem} which can be visualized with the help of Figure \ref{fig: W-Blocks}. 
 In particular, definitions of $W$, $Y$, $P_h$ and $P_v$ are used.
 
 Three lemmas and the final results are presented. The first lemma gives a lower bound on the number of consecutive contributing diagonals and, provided a horizontal or vertical pair intersects $W$, the second gives upper bounds on vertical and horizontal pairs, respectively. These results are leveraged against the third lemma, a relatively unconditional and straightforward upper bound on the total number of horizontal and vertical pairs, to yield Theorem \ref{thm: upperbound}. This theorem gives $\rb(\mn,eq) = m+n+1$ for all $8\leq m\leq n$. Along with several earlier statements to cover the cases of $2\leq m \leq 7$, Theorem \ref{thm: upperbound} subsequently determines the final result and the result claimed to be true in the exposition of this paper, Theorem \ref{thm:final}. 
\begin{lem}\label{lem:consecutivecontributing}
 If $c$ is an exact, rainbow-free $(m+n+1)$-coloring of $\mn$ for $eq$ with $3 \le m\le n$, then there are at least $m+n-2\log_2(m/s_2)-2$ pairs of consecutive, contributing off-diagonals.

\end{lem}

\bpf Define $\ell = |c(D_m)|$ and notice that
   Corollary  \ref{cor:contributingoffdiag} gives at least $m+n - \log_2(m/s_2) -1$ contributing off-diagonals and Lemma \ref{lem:ellminusthree} gives at most $\ell - 3$ non-contributing off-diagonals.  If all of the contributing off-diagonals were consecutive, then there would be $m+n-\log_2(m/s_2) - 2$ pairs of consecutive, contributing off-diagonals.  
   However, every non-contributing diagonal, including every non-contributing off-diagonal and $D_m$, when inserted between two consecutive, contributing off-diagaonals, can reduce the count of pairs of consecutive, contributing off-diagonals by at most one.  
   Finally, since Lemma \ref{lem:s2} implies $\ell \le \log_2(m/s_2) +2$, there are at least \[(m+n - \log_2(m/s_2)  - 2) -(\ell - 2) \ge m+n-2\log_2(m/s_2)-2\] pairs of consecutive, contributing off-diagonals. \epf

\begin{lem}\label{upperbound for vert/hor pairs}
   Let $c$ be an exact, rainbow-free $(m+n+1)$-coloring of $[m]\times [n]$ for $eq$ with $4\le m \le n$.  If there is a horizontal pair $P_h$ intersecting $W$, then there are at most $2s_2-2$ vertical pairs in $\mn$. Likewise, if there is a vertical pair $P_v$ intersecting $W$, then there are at most $2s_2-2$ horizontal pairs in $\mn$.
\end{lem}

\bpf Note that if $|c(D_m)| = 3$ then $c$ is not rainbow-free by Theorem \ref{thm:main3colors} so assume $|c(D_m)| \ge 4$.  Then, Lemma \ref{lem:s2} gives $s_2 \le m/4$.  Further, Lemma \ref{lem:consecutivecontributing} implies at least $m+n - 2\log_2(m/s_2) - 2$ pairs of consecutive, contributing off-diagonals.  At most $2(2s_2 - 2)$ of these pairs are entirely contained in $Y$ so at least $m+n - 2\log_2(m/s_2) - 2 - 2(2s_2 -2)$ pairs of consecutive, contributing off-diagonals that intersect $W$.  Using $4\le m\leq n$, $m/s_2 \leq m/2$, $4s_2 \leq m$ and $\log_2(x)$ is increasing gives 

   \[ m+n - 2\log_2(m/s_2) - 2 - 2(2s_2 -2) \ge  m - 2\log_2(m) + 4 \ge4. \]
   
That is, there are at least $4$ pairs of consecutive, contributing off-diagonals that intersect $W$.  Further, Theorem \ref{thm:nojumpsever} indicates that these pairs are either horizontal or vertical.  Thus, if any of these pairs that intersect $W$ is a horizontal pair, there cannot also be a vertical pair that intersects $W$ without creating a contributing disjoint corner which, by Lemma \ref{lem: hor-vert-pairs}, implies $c$ is not rainbow-free.  Thus, all vertical pairs must be contained completely within $Y$.  Only half of the pairs in $Y$ can be vertical so, in this case, there are are most $2s_2 - 2$ vertical pairs in $\mn$.  A similar result is obtained when it is assumed that a vertical pair intersects $W$.\epf

\begin{lem}\label{lem:max hor-vert pairs} 
    Let $c$ be an exact, rainbow-free (m+n+1)-coloring of $\mn$ for $eq$ with $3\le m\le n$.  Then, there are at most $n-1$ possible horizontal pairs, and there are at most $m-1$ vertical pairs.
\end{lem}

\bpf
 Lemma \ref{onedistinctcolorlem} and Theorem \ref{thm:nojumpsever} imply that two distinct horizontal pairs cannot both intersect the $i$th and $(i+1)$th columns, and two distinct vertical pairs cannot both intersect the $j$th and $(j+1)$th rows. Since there are $n-1$ pairs of consecutive columns and $m-1$ pairs of consecutive rows, the desired result is obtained.
\epf

\begin{thm}\label{thm: upperbound}
    
     If $c$ is an exact $(m+n+1)$-coloring of $\mn$ with $8\le m \le n$, then $\mn$ contains a rainbow solution to $eq$.  
\end{thm}

\begin{proof}
    Assume, for the sake of contradiction, that $c$ is rainbow-free.  By Theorem \ref{thm:main3colors}, $|c(D_m)| \ge4$. Now, Lemma \ref{lem:s2} implies $4 \leq \log_2(m/s_2)+2$ so that $m \geq 4s_2$.
    \begin{description}
        \item[Case 1.] There exists a horizontal pair intersecting $W$. 
        \par
        By Lemma \ref{lem:consecutivecontributing}, there exist at least $m+n-2\log_2(m/s_2)-2$ pairs of consecutive contributing off-diagonals. By Theorem \ref{thm:nojumpsever}, there are at least $m+n-2\log_2(m/s_2)-2$ consecutive contributing pairs of elements which each must be a vertical or horizontal pair. Since Lemma \ref{upperbound for vert/hor pairs} implies there at most $2s_2-2$ vertical pairs, there must be at least $m+n-2\log_2(m/s_2)-2s_2$ horizontal pairs in $\mn$.
        \par
        Using $m/s_2 \le m/2$, $s_2\le m/4$, $8\le m$ and that $\log_2(x)$ is increasing, it follows that \[m + n - 2\log_2(m/s_2) - 2s_2 \geq n + m/2 - 2\log_2(m) + 2 \geq n > n-1,\]
        contradicting Lemma \ref{lem:max hor-vert pairs}.

        \item[Case 2.] There exists a vertical pair intersecting $W$. 
        \par
        Using an argument similar to the first case shows that there must be at least $m+n-2\log_2(m/s_2)-2s_2$ vertical pairs $\mn$. Again, \[m + n - 2\log_2(m/s_2) - 2s_2 > n-1 \ge m-1,\] contradicting Lemma \ref{lem:max hor-vert pairs}. 
    \end{description}
 Since both cases give a contradiction, $c$ is not rainbow-free.
\end{proof}

Combining Lemma \ref{mnlower}, Corollary \ref{cor:m=2}, Theorem \ref{thm:main3colors}, and Theorem \ref{thm: upperbound} gives Theorem \ref{thm:final}.

\begin{thm}\label{thm:final}

   If $2\le m \le n$ , then $\rb(\mn,eq) = m+n+1$.

\end{thm}

\bpf First, Corollary \ref{cor:m=2} gives the desired result when $m=2$.  Now, Lemma \ref{mnlower} gives that $m+n+1 \le \rb(\mn,eq)$ for $3\le m \le n$.  For $3\le m \le 7$, at most three colors are in the main diagonal so Theorem \ref{thm:main3colors} gives $\rb(\mn,eq) = m+n+1$.  Finally, for $8\le m \le n$, Theorem \ref{thm: upperbound} gives $\rb(\mn,eq) = m+n+1$.\epf

\section*{Acknowledgements} 
Thank you to the University of Wisconsin-La Crosse (UWL) Deans Distinguished Fellows program that supported the third and sixth authors.  Thanks also to UWL Undergraduate Research and Creativity grants that supported the first, third and sixth authors.  Finally, thanks to UWL Department of Mathematics and Statistics Bange/Wine Undergraduate Research Endowment that supported the third author.


\begin{thebibliography}{20}

    \bibitem{ATHNWY} K. Ansaldi, H. El Turkey, J. Hamm, A. Nu'Man, N. Warnberg and  M. Young.  Rainbow Numbers of $\mathbb{Z}_n$ for $a_1x_1+a_2x_2 + a_3x_3 = b$.  \emph{INTEGERS}, {\bf 20}, A51 (2020). \url{http://math.colgate.edu/~integers/vol20.html}

    \bibitem{AF} M. Axenovich and D. Fon-Der-Flaass, On rainbow arithmetic progressions, \emph{Electron. J. Combin.} {\bf 11} (2004), no. 1, Research Paper 1, 7pp. \url{https://www.combinatorics.org/ojs/index.php/eljc/article/view/v11i1r1}
		
	\bibitem{SWY} Z. Berikkyzy, A. Schulte, E. Sprangel, S. Walker, N. Warnberg and M. Young. Anti-van der Waerden Numbers on Graphs. \emph{Graphs and Combinatorics}, {\bf 38(124)} (2022). \url{https://doi.org/10.1007/s00373-022-02516-9}

    \bibitem{BSY} Z. Berikkyzy, A. Schulte, and M. Young, Anti-van der Waerden numbers of 3-term arithmetic progressions, \emph{Electron. J. Combin.} {\bf 24} (2017), no. 2, Paper 2.39, 9 pp. \url{https://www.combinatorics.org/ojs/index.php/eljc/article/view/v24i2p39}
		
	\bibitem{BKKTTY} E. Bevilacqua, A. King, J. Kritschgau, M. Tait, S. Tebon and M. Young, Rainbow numbers for $x_1 + x_2 = kx_3$ in $\mathbb{Z}_n$, \emph{INTEGERS}, {\bf 20}, A50 (2020). \url{http://math.colgate.edu/~integers/vol20.html}
		
    \bibitem{DMS} S. Butler, C. Erickson, L. Hogben, K. Hogenson, L. Kramer, R.L. Kramer, J. Lin, R.R. Martin, D. Stolee, N. Warnberg and M. Young, Rainbow Arithmetic Progressions, \emph{J. Comb.} {\bf 7(4)} (2016), 595--626.\url{https://www.intlpress.com/site/pub/pages/journals/items/joc/content/vols/0007/0004/a003/index.php}

    \bibitem{FGRWW} K. Fallon, C. Giles, H. Rehm, S. Wagner and N. Warnberg, Rainbow numbers of $[n]$ for $\sum_{i=1}^{k-1} x_i = x_k$, \emph{Austral. J. Combin.} {\bf 77(1)} (2020), 1--8.\url{https://ajc.maths.uq.edu.au/pdf/77/ajc_v77_p001.pdf}

    \bibitem{RFC} M. Huicochea and A. Montejano, The Structure of Rainbow-Free Colorings For Linear Equations on Three Variables in $\mathbb{Z}_p$, \emph{Integers} {\bf 15A} (2015), A8. \url{http://math.colgate.edu/~integers/vol15a.html}

    \bibitem{J}	V. Jungi\'c, J. Licht (Fox), M. Mahdian, J. Ne\u{s}etril, and R. Radoi\u{c}i\'c, Rainbow arithmetic progressions and anti-Ramsey results, \emph{Combinatorics, Probability and Computing} {\bf 12} (2003), no 5-6, 599--620. \url{https://doi.org/10.1017/S096354830300587X}
		
	\bibitem{LM} B. Llano and A. Montejano, Rainbow-free Colorings of $x+y = cz$ in $\mathbb{Z}_p$, \emph{Discrete Math.} {\bf 312}, (2012), 2566--2573. \url{https://doi.org/10.1016/j.disc.2011.09.005}
		
	\bibitem{MW} J. Miller and N. Warnberg, Anti-van der Waerden Number of Graph Products of Cycles, \url{https://arxiv.org/abs/2205.11621}.
		
	\bibitem{RSW} H. Rehm, A. Schulte and N. Warnberg, Anti-van der Waerden numbers on Graph Products, \emph{Austral. J. Combin.} {\bf 73(3)} (2019), 486--500. \url{https://ajc.maths.uq.edu.au/pdf/73/ajc_v73_p486.pdf}
		
	\bibitem{U} K. Uherka, An introduction to Ramsey theory and anti-Ramsey theory on the integers, Master's Creative Component (2013), Iowa State University.

    \bibitem{finabgroup} M. Young.  Rainbow Arithmetic Progressions in Finite Abelian Groups, \emph{J. Comb.} {\bf 9(4)} (2018), 619--629. \url{https://www.intlpress.com/site/pub/pages/journals/items/joc/content/vols/0009/0004/a003/index.php}
        
\end{thebibliography}
\end{document}